\documentclass[11pt]{amsart}
\usepackage[top=3cm, bottom=2.5cm, left=2.5cm, right=2.5cm]{geometry}                
\usepackage[english]{babel}

\usepackage{tcolorbox}
\usepackage{rotating}

\usepackage{amssymb}
\usepackage{tikz}
\usetikzlibrary{3d}
\usepackage{tikz-3dplot}
\usepackage[colorlinks=true]{hyperref}
\usepackage{enumerate}
\usepackage{mathrsfs}
\usepackage[all]{xy}
\usepackage[normalem]{ulem}
\usepackage{float} 
\usepackage{todonotes}

\usepackage{mathpazo}
\usepackage[scaled=.95]{helvet}
\usepackage{courier}

\usepackage{caption,subfig}
\usepackage{ulem}

\numberwithin{equation}{section}
\numberwithin{figure}{section}
\numberwithin{table}{section}

\newtheorem{theorem}{Theorem}[section]
\newtheorem{proposition}[theorem]{Proposition}
\newtheorem{lemma}[theorem]{Lemma}
\newtheorem{corollary}[theorem]{Corollary}

\theoremstyle{definition}
\newtheorem{definition}[theorem]{Definition}

\theoremstyle{remark}
\newtheorem{remark}[theorem]{Remark}
\newtheorem{example}[theorem]{Example}

\DeclareMathOperator{\Spec}{\textnormal{Spec}}

\DeclareMathOperator{\Ht}{\textnormal{Ht}}

\DeclareMathOperator{\supp}{Supp}

\DeclareMathOperator{\syz}{\textnormal{Syz}}

\newcommand{\kk}{\mathbb{K}}

\newcommand{\ZZ}{\mathbb{Z}}
\newcommand{\T}{\mathbb{T}}

\newcommand{\reg}{\textnormal{reg}}
\newcommand{\ee}{\boldsymbol{\sf e}}
\newcommand{\ff}{{\boldsymbol{\sf f}}}
\newcommand{\dd}{\boldsymbol{d}}
\newcommand{\mf}[2]{f_{#1, #2}}
\newcommand{\mee}[2]{\ee_{#1}^{#2}}
\newcommand{\msyz}[2]{{\sf S}_{#1}^{#2}}
\newcommand{\pC}[4]{C_{#1,#2}^{#3,#4}}
\newcommand{\pB}[6]{B_{#1,#2,#3}^{#4,#5,#6}}
\newcommand{\ridast}[1]{\buildrel {#1}^{\ast}\over \longrightarrow}
\newcommand{\cridast}[1]{\buildrel {#1}^{\ast}\over \longrightarrow_\star}

\newcommand{\MFFunctor}[1]{\underline{\mathbf{Mf}}_{#1}}
\newcommand{\MFScheme}[1]{\mathbf{Mf}_{#1}}
\newcommand{\MFSchemeMin}{\mathbf{T}}
\newcommand{\MSyzFunctor}[1]{\underline{\mathbf{Mf}}^{\textsc{syz}}_{#1}}
\newcommand{\MSyzScheme}[1]{\mathbf{Mf}^{\textsc{syz}}_{#1}}
\newcommand{\MResFunctor}[1]{\underline{\mathbf{Mf}}^{\textsc{res}}_{#1}}
\newcommand{\MResScheme}[1]{\mathbf{Mf}^{\textsc{res}}_{#1}}
\newcommand{\MSyzFunctorDue}[1]{\underline{\mathbf{Mf}}^{\pi_2}_{#1}}
\newcommand{\MSyzSchemeDue}[1]{\mathbf{Mf}^{\pi_2}_{#1}}

\newcommand{\PP}[1]{\mathcal P_{#1}}

\newcommand{\HTemph}[1]{\uwave{#1}}

\newtcbox{\htbox}{on line,colback=black!35,colframe=black, arc=0pt,outer arc=0pt,boxsep=0pt,left=2pt,right=2pt,top=2pt,bottom=2pt,
boxrule=0.25pt,bottomrule=0.25pt,toprule=0.25pt,inherit height}

\begin{document}

\title{Free resolutions and marked families}

\author[C.~Bertone]{Cristina Bertone}
\address{Dipartimento di Matematica \lq\lq G.~Peano\rq\rq\\ Universit\`a degli Studi di Torino\\ 
         Via Carlo Alberto 10\\ 10123 Torino\\ Italy.}
\email{\href{mailto:cristina.bertone@unito.it}{cristina.bertone@unito.it}}
\urladdr{\url{https://sites.google.com/view/cristinabertone/}}

\author[F.~Cioffi]{Francesca Cioffi}
\address{Dipartimento di Matematica e Applicazioni \lq\lq R. Caccioppoli\rq\rq \\ Universit\`a degli Studi di Napoli Federico II\\ 
         Complesso Universitario di Monte S. Angelo\\ Via Cintia \\ 80126 Napoli \\ Italy.}
\email{\href{mailto:cioffifr@unina.it}{cioffifr@unina.it}}

\author[P.~Lella]{Paolo Lella}
\address{Dipartimento di Matematica\\ Politecnico di Milano\\ 
         Piazza Leonardo da Vinci 32\\ 20133 Milano\\ Italy.}
\email{\href{mailto:paolo.lella@polimi.it}{paolo.lella@polimi.it}}
\urladdr{\url{http://www.paololella.it/}}



\renewcommand{\baselinestretch}{1.2}

\begin{abstract}
Let $\kk$ be a field and $A$ a Noetherian $\kk$-algebra. In a paper of 2020, M.~Albert, C.~Bertone, M.~Roggero and W.~M.~Seiler proved that, given a quasi-stable module $U \subset R^m$ with $R=\kk[x_0,\dots,x_n]$, any submodule $M\subseteq (R\otimes A)^m$ generated by a marked basis over $U$ admits a special free resolution described in terms of marked bases as well, called the {\em $U$-resolution of $M$}. In this paper, we first investigate the minimality of the $U$-resolution and its structure. When $M$ is an ideal and $A=\kk$, we show that $M$ is componentwise linear if and only if its $U$-resolution is minimal, up to a linear change of variables. Then, adopting a functorial approach to the construction of the $U$-resolution, we prove that certain functors naturally associated with the resolution are  isomorphic. These isomorphisms arise from the fact that the marked basis of the $i$-th syzygy module in the $U$-resolution can be expressed in terms of the coefficients of the marked basis of~$M$. Moreover, when $M$ is an ideal of depth at least 2, this correspondence can be reversed: in this case, the marked basis of $M$ itself can be written in terms of the coefficients of the marked basis of its first syzygy module.
\end{abstract}

\maketitle

\section{Introduction}
Let $\kk$ be a field, and let $R = \kk[x_0, \dots, x_n]$ denote the polynomial ring in $n+1$ variables with the standard grading.
Consider a (monomial) quasi-stable submodule $U$ of the free graded $R$-module $R^m(-\boldsymbol{d})$, where $\boldsymbol{d} = (d_1, \dots, d_m) \in \mathbb{Z}^m$.
For every Noetherian $\kk$-algebra $A$, we can consider a set of homogeneous marked elements in $R_A^m(-\boldsymbol{d}) = (R \otimes_\kk A)^m(-\boldsymbol{d})$ over a special monomial set generating $U$, called the Pommaret basis of $U$.
Let $M$ be the submodule in $R_A^m (- \dd)$ generated by this marked set (we omit \lq\lq over $U$\rq\rq~when not strictly needed). If for every $s\geqslant 0$, the $A$-module  $R_A^m (- \dd)_s$ is the direct sum of the $A$-submodule generated by the degree-$s$ part of $M$ and the  $A$-submodule generated by the terms of degree $s$ outside $U$, we say that the marked set generating $M$ is a marked basis. 

Over the past fifteen years, marked bases have been intensively studied in a series of papers (for example, see \cite{CR11, BCLR13, BLR13, LR2, CMR2015, macaulay} and the references therein), especially in the setting of polynomial ideals. They have proven to be an effective tool for both theoretical and  computational investigations related to the study of Hilbert schemes. More recently, in the paper \cite{ABRS}, the study of marked bases and of their functorial features has been extended to submodules of free graded modules with applications to Quot schemes as well. 

Starting from analogous contributions for polynomial ideals,
in \cite[Theorem 5.1]{ABRS} it was proved  that  the family  of submodules 
having a marked basis over a quasi-stable module $U$, denoted by $\MFFunctor{U}(A)$, can be described by means of polynomial equations whose indeterminates are the $x$-coefficients in a generic marked set over $U$. This construction is functorial (see \cite{LR2}), and the functor associating to each Noetherian $\kk$-algebra $A$ the set $\MFFunctor{U}(A)$ is representable and represented by an affine scheme denoted by $\MFScheme{U}$.
  
Furthermore, a module generated by a marked basis over the quasi-stable module $U$ has a free resolution similar to that for stable ideals given in \cite{EK}, that we call $U$-resolution (see Theorem \ref{thm:freeRes}, which was stated and proved in \cite{ABRS}).

The equations defining the marked scheme $\MFScheme{U}$ are obtained by some features of marked bases, which are interesting from an algorithmic and algebraic point of view. Indeed, there are several parallels between Gr\"obner bases and marked bases over quasi-stable modules, with the key distinction that the theory of marked bases and its associated reduction algorithms are term-order free. The term order that allows to have a Noetherian reduction process with Gr\"obner bases is replaced by properties of the quasi-stable ideal given by its Pommaret basis: the disjoint cones decomposition of the set of terms in $U$, and the existence of an induced well-founded order on $\T$ that makes a marked basis over a quasi-stable module an ordered structure (see \cite{CMR2019}).

As already observed, from the geometric point of view
marked bases proved to be a powerful tool to investigate Quot and Hilbert schemes. More precisely, if we consider a truncated quasi-stable module $U$, with truncation degree sufficiently large, and $p(t)$ is the Hilbert polynomial of $R^m(-\dd)/U$,
then the family of marked bases over $U$ is an open subset of the Quot scheme that parametrizes quotients of $R_A^m(-\dd)$ having Hilbert polynomial $p(t)$. Indeed, $p(t)$ is also the Hilbert polynomial of  $R_A^m(-\dd)/M$, for every $M$ generated by a marked basis over the quasi-stable module $U$. Furthermore, up to some change of coordinates, these open subsets cover the whole Quot scheme \cite[Propositions 10.1, 10.2, 10.3]{ABRS}. This open covering of the Hilbert scheme has enabled the study of several meaningful loci, as well as investigations of Hilbert schemes over bases other than $\mathbb{P}^n$, and even the disproof of conjectures (see \cite{SmoothGore,BCOS, GGGL} and references therein).

Nevertheless, two modules $M_1$ and $M_2$ belonging to $\MFFunctor{U}(A)$, with $U$ a  quasi-stable module, might be quite different. Indeed, although they both have a $U$-resolution, the respective minimal free resolution can remarkably differ; in general, Betti numbers, Castelnuovo-Mumford regularity and projective dimension are not the same for the two modules, although the last two invariants are upper bounded by those of $U$ (see  Corollary \ref{cor:maggiorazioni} and the example in Section~\ref{sec:finalEx}). All this implies for instance that, if we consider $m=1$, the projective schemes in $\mathbb P^n$ defined by $M_1$ and $M_2$ are geometrically quite different, although $M_1$ and $M_2$ have the same Hilbert polynomial.

\medskip 

Motivated by the need to better understand the properties of modules belonging to $\MFFunctor{U}(A)$, in this paper we investigate the syzygies of marked bases over quasi-stable modules.
 We exhibit a set of generators for the first syzygies of a marked basis and we prove that it is possible to predict the support of these syzygies (see Corollaries \ref{cor:suppSyz} and \ref{cor:forecastNoTerms}). Since it is possible to have a marked basis on a suitable stable module for the first syzygies of a marked basis over a quasi-stable module (see Theorem \ref{th:fundamental}), we can iterate the process and get a marked basis for the $t$-th module appearing in the $U$-resolution of a module in $\MFFunctor{U}(A)$, forecasting the support of the elements in the marked basis of the $t$-th syzygies.

In Section \ref{sec:minimal-resolution} we study  whether the $U$-resolution of a module $M$ in $\MFFunctor{U}(A)$ is minimal. Indeed, even if the $U$-resolution of $U$ is non-minimal, in $\MFFunctor{U}(A)$ we can find modules having minimal $U$-resolution, see Example \ref{ex:minimality-quasi-stable}. We show that if the matrix of the $i$-th morphism in a $U$-resolution contains no non-zero constant entries (with respect to the variables $x_i$), then the same property holds for the matrices of all subsequent morphisms (Theorem \ref{prop:firstSyzForMin}). This means that, in the case $A=\kk$, if the marked basis of the $i$-th syzygies in a $U$-resolution is a minimal generating set, then also the marked bases of the following syzygy modules are minimal. In particular, the $U$-resolution is minimal if and only if the $U$-marked basis of the module $M$ is a minimal generating set of it (Corollary \ref{cor:minimality-betti-numbers}).
Since every Gr\"obner basis whose initial module is quasi-stable is also a marked basis in our setting, we can then easily conclude that the marked bases whose $U$-resolution is minimal (when $U$ is not stable) are not Gr\"obner bases for any term order (Proposition \ref{prop:term order}).

We further prove that the minimality of the $U$-resolution of $M \in \MFFunctor{U}(\kk)$ is equivalent (up to a general linear change of coordinates) to the property that $M$ is componentwise linear (Theorems \ref{thmn:minimal-compwLinear} and \ref{thmn:compwLinear-minimal}). This generalizes \cite[Theorem 8.2]{Seiler2009II} and \cite[Theorem~1.1]{AHH} to marked bases over a field in any characteristic.

In Section \ref{sec:SyzFunct}, following the functorial approach to the study of marked schemes, see for instance \cite{LR2}, we define the $U$-marked  syzygy functor and the $U$-marked resolution functor (Definition \ref{def:FuncSyzRes}).
It is immediate from their definitions, by Yoneda's Lemma, that both the scheme representing the marked syzygy functor and that representing the marked resolution functor can be isomorphically canonically projected on $\MFScheme{U}$ (Corollary~\ref{cor:isomorfismo}). We further  give an explicit construction of the affine schemes representing these functors (Theorems \ref{thm:reprFuncSyz} and \ref{prop:elimSyzn}), similarly to the construction of $\MFScheme{U}$. Alternatively, we can define the scheme representing the $U$-marked  syzygy functor imposing that the multiplication of the two matrices representing the first two morphisms in a generic $U$-resolution is $0$ (Corollary \ref{cor:prodotto matrici}), and apply the same argument for the $U$-marked resolution functor.

Finally, we study the projection of the scheme representing the marked syzygy functor onto the other factors. Under the hypothesis that $U$ is a monomial ideal having depth at least~$2$, we prove  that the marked syzygy functor can be also canonically and isomorphically projected to the $U$-marked second projection functor, which is the functor represented by the scheme parametrizing marked bases of fundamental syzygies of $U$ (see Definition \ref{def:second factor} and Theorem~\ref{th:isomorphism 2}). The proof relies on results in \cite{BE1974}. 

Corollary~\ref{cor:isomorfismo} and Theorem~\ref{th:isomorphism 2} relate the parameters appearing as coefficients of the syzygies and those appearing as coefficients of the polynomials in a marked basis.  The first isomorphic projection allows the syzygy coefficients to be expressed in terms of the polynomial coefficients, while the second isomorphic projection, under certain conditions, allows the reverse.

There are several examples throughout the paper, and we conclude by working out a complete example in Appendix \ref{sec:finalEx}. In this example, we describe how to characterize the locally Cohen-Macaulay curves of $\mathbf{Hilb}^{3t+2}(\mathbb{P}^3)$ by considering the $J$-resolution of a generic ideal in $\MFScheme{J}$ and studying the Hartshorne-Rao module, for a suitable quasi-stable ideal $J$. Along this example, we use the fact that Corollary \ref{cor:isomorfismo} ensures that we can rewrite the parameters appearing as coefficients in a generic marked basis of the first syzygies of a marked set in terms of the parameters appearing as coefficients in the marked set.


\section{Marked bases and their syzygies}
\label{sec:marked bases and syzygies}

Let $\kk$ be a field and $A$  be a Noetherian $\kk$-algebra with $1_A=1_{\kk}$. We consider the polynomial ring $R=\kk[x_0,\dots,x_n]$ and $R_A:=R\otimes_{\kk}A=A[x_0,\dots,x_n]$, with $x_0<\dots<x_n$. A {\em term} in $R$ is a power product $x^{\underline{\alpha}} := x_0^ {\alpha_0}\cdots x_n^{\alpha_n}$, where $\underline{\alpha} = (\alpha_0,\ldots,\alpha_n) \in \ZZ^{n+1}_{\geqslant 0}$, and the set of all terms in $R$ is denoted by $\T$. For every $x^{\underline {\alpha}} \in \T$, $x^{\underline {\alpha}} \neq 1$, we denote by $\max(x^{\underline {\alpha}})$ the largest variable dividing $x^{\underline {\alpha}}$ and by $\min(x^{\underline {\alpha}})$ the smallest variable dividing it. The variables less than or equal to $\min (x^{\underline {\alpha}})$ are called the {\em multiplicative variables of $x^{\underline {\alpha}}$}. The remaining variables are the {\em non-multiplicative variables of $x^{\underline {\alpha}}$}.

We use the standard grading on $R_A$, i.e.~$\deg(x_j)=1$, for all $j\in \{0,\dots,n\}$, and $\deg(a)=0$, for all $a\in A$. Thus, $\deg(x^{\underline {\alpha}})=\vert \underline{\alpha}\vert=\sum \alpha_i$.

Let $m$ be a positive integer and $\boldsymbol d = (d_1,\dots,d_m)\in \mathbb Z^m$. More generally, we consider the free graded $R_A$-module $R_A^m (- \dd):=\bigoplus_{i=1}^m R_A (-d_i)\ee_i$, with standard basis $\{\ee_1,\dots,\ee_m\}$. If $\dd=(0,\dots,0)$, we simply write $R_A^m$ instead of $R_A^m (- \dd)$. 
A term of $R_A^m (- \dd)$ is an element of the form $x^{\underline {\alpha}}\, \ee_i$, where $i\in\{1\,\dots,m\}$ and $x^{\underline {\alpha}} \in \T$. The degree of a term $x^{\underline {\alpha}}\, \ee_i$ is $\deg(x^{\underline {\alpha}}\, \ee_i)=\vert \underline{\alpha}\vert+d_i$.

The set of the terms in $R_A^m (- \dd)$ is denoted by $\T^m(- \dd)$, or simply by $\T^m$ if the grading is irrelevant.

If $S$ is a  subset of $R_A^m (- \dd)$, we denote by $\langle S\rangle$ the $A$-module generated by $S$, and by $(S)$ the $R_A^m (- \dd)$-module generated by $S$. Furthermore, for every $t\in \mathbb Z$, we denote by $S_{t}$ the subset of $S$ consisting of homogeneous elements of degree $t$. Throughout this paper, 
 we consider graded finitely generated modules $M$ of  $R_A^m (- \dd)$, so that $M=\bigoplus_{t \in \mathbb Z}M_{t}$. We denote by $M_{\geqslant m}$ and $M_{\leqslant m}$ the direct sums $M=\bigoplus_{t \geqslant m }M_{t}$ and $M=\bigoplus_{t \leqslant m }M_{t}$.

For every element $f\in R_A^m (- \dd)$, we denote by $\supp(f)$ the set of terms appearing in $f$ with non-zero coefficient. 

A submodule $U \subset R_A^m (- \dd)$ is {\em monomial} if generated by terms in $\T^m(- \dd)$. A monomial submodule $U$ can be written as
\begin{equation}\label{eq:monidU}
U= \bigoplus_{k=1}^m J^{(k)} \ee_k
\end{equation}
where $J^{(k)}$ is the monomial ideal of $R_A$ generated by the terms $x^{\underline {\alpha}}$ such that $x^{\underline {\alpha}}\, \ee_k$ belongs to $U$.  A monomial module $U$ has a unique minimal set of term generators, called \emph{the monomial basis} of $U$, denoted by $\mathcal B_U$. More precisely, considering $U$ as in \eqref{eq:monidU},  $\mathcal B_U$ is the set $\bigcup_{k=1}^m  \mathcal B_{J^{(k)}}\ee_k$, where $\mathcal B_{J^{(k)}}$ is the minimal monomial basis of $J^{(k)}$. 
We denote by  $\mathcal N(U)\subseteq \T^m(-\dd)$ the set of terms not in $U$, called the {\em sous-escalier} of $U$. Note that  $\mathcal N(U)= \bigcup_{k=1}^m \mathcal N(J^{(k)}) \ee_k$.

\begin{definition} {\rm (\cite[Defintion 3.1]{ABRS}, \cite[Definition 2.1 and Example 2.4]{Seiler2009I})}
For each term $x^{\underline {\alpha}}$ in $\T$, we define its \emph{Pommaret cone} in $R_A$: if $x^{\underline{\alpha}}=1$, set $\mathcal C(1):=\T$; if $x^{\underline{\alpha}}\neq 1$, then the Pommaret cone of $x^{\underline{\alpha}}$ is
\[
\mathcal C(x^{\underline {\alpha}}):=
\{x^{\underline \delta} x^{\underline \alpha} \mid \delta_i=0, \, \forall \ x_i > \min(x^{\underline {\alpha}}) \}\subset \T .
\]
For every term $x^{\underline {\alpha}}\, \ee_k$ in $\T^m$, the \emph{Pommaret cone} of $x^{\underline {\alpha}}\, \ee_k$ in $R_A^m$ is the set
\[
\mathcal C(x^{\underline {\alpha}}\, \ee_k):=\mathcal C(x^{\underline {\alpha}})\ee_k= \left \{ x^{\underline{\delta}} x^{\underline {\alpha}} \, \ee_k\ \middle\vert\ x^{\underline{\delta}} x^{\underline {\alpha}} \in \mathcal C(x^{\underline {\alpha}})\right\}\subset \T \ee_k.
\]

A set $T$ of terms generating a monomial submodule $U$  in $R_A^m (- \dd)$ is called a {\em Pommaret basis} of $U$ if 
\[
U\cap \T^m = \bigsqcup_{x^{\underline {\alpha}}\, \ee_k \in T} \mathcal C(x^{\underline {\alpha}}\, \ee_k)
\]
\end{definition}

If a monomial submodule $U$ has a Pommaret basis, it is unique and denoted by $\PP{U}$.
    
\begin{definition}\label{def:qsIdeal} \cite[Definition 3.2 and Theorem 3.3]{ABRS}
A \emph{quasi-stable} module is a monomial module with a {\em finite} Pommaret basis, denoted by $\PP{U}$.  A monomial module $U$ is \emph{stable} if it is quasi-stable and its Pommaret basis $\PP{U}$ coincides with its monomial basis $\mathcal B_U$. 
\end{definition}

\begin{example}\label{ex:section2-part1}
The ideal $J = (x_2^2,x_1^2) \subset \kk[x_0,x_1,x_2]$ is quasi-stable, but not stable, because $x_2x_1^2 \in J$ while $x_2x_1^2 \notin \mathcal{C}(x_2^2) \sqcup \mathcal{C}(x_1^2)$. Nevertheless, $J\cap \T = \mathcal{C}(x_2^2) \sqcup\mathcal{C}(x_1^2) \sqcup \mathcal{C}(x_2x_1^2)$, hence $\mathcal{P}_{J} = \mathcal{B}_{J} \cup \{x_2x_1^2\}$. 

The ideal $J' = (x_2^2,x_2x_1,x_1^3)  \subset \kk[x_0,x_1,x_2]$ is stable, and $\mathcal{P}_{J'} = \mathcal{B}_{J'}$.

The ideal $(x_1^2,x_0^2) \subset \kk[x_0,x_1,x_2] $ is not quasi-stable.
\end{example}

\begin{definition}{\rm (\cite{RS}, \cite[Definition 4.2]{ABRS})}
A marked  element $\mf{\underline{\alpha}}{k}$ in $R_A^m (- \dd)$ is a homogeneous element with a distinguished term $x^{\underline {\alpha}}\, \ee_k\in\supp(\mf{\underline{\alpha}}{k})$ with coefficient $1_A$ in $\mf{\underline{\alpha}}{k}$. This term is called the \emph{head term} of $\mf{\underline{\alpha}}{k}$, denoted by $\Ht(\mf{\underline{\alpha}}{k})$, and highlighted with wavy underline $\HTemph{x^{\underline{\alpha}}\, \ee_k}$.
\end{definition}

\begin{definition}\label{def:MarkedSet} \cite[Definition 4.3]{ABRS} 
Let $U\subset R_A^m (- \dd)$ be a quasi-stable module. A {\em $U$-marked set} is a finite set $F \subset R_A^m (- \dd)$ of marked elements $\mf{\underline{\alpha}}{k}$ with pairwise distinct head terms $\Ht(\mf{\underline{\alpha}}{k}) =x^{\underline {\alpha}}\, \ee_k \in \PP{U}$, such that $\supp(\mf{\underline{\alpha}}{k}-x^{\underline {\alpha}}\, \ee_k) \subseteq  \mathcal N(U) $ and $\{\Ht(\mf{\underline{\alpha}}{k})\ \vert\ \mf{\underline{\alpha}}{k}\in F\} = \PP{U}$. 
A $U$-marked set $F$ is a {\em $U$-marked basis} of the module $(F)$ if, for all $s$, $\left(R_A^m (- \dd)\right)_s = ( F )_s \oplus \langle \mathcal N(U)_s \rangle$.
\end{definition}

Note that if a module is generated by a $U$-marked basis, then this basis is unique.

\begin{definition}\label{def:riscrittura}\cite[Definition 4.7]{ABRS}
Let $U\subseteq R_A^m (- \dd)$ be a quasi-stable module and $F$ a $U$-marked set.
We denote by $\ridast{F}$ the transitive closure of the relation $h \longrightarrow h -\lambda\, x^{\underline{\eta}}\, \mf{\underline{\alpha}}{k}$ among homogeneous elements of the same degree, where $x^{\underline{\eta}} x^{\underline {\alpha}}\, \ee_k$ is a term in $h$ with a non-zero coefficient $\lambda \in A$ and $x^{\underline{\eta}} x^{\underline {\alpha}} \,\ee_k \in \mathcal{C}(x^{\underline{\alpha}}\,\ee_k)$. We write $h \cridast{F} g$ if $h \ridast{F}g$ and $\supp(g)\subset \mathcal N(U)$; then $g$ is called the \emph{$F$-normal form} of $f$.
\end{definition}

The relation in Definition \ref{def:riscrittura} is Noetherian and confluent (see \cite[Propositions 4.8 and 4.11]{ABRS}). 

\bigskip 

From now on, let $U\subset R_A^m (- \dd)$ be a quasi-stable module with Pommaret basis $\PP{U}$. Let $F=\{\mf{\underline{\alpha}}{k}\}_{x^{\underline {\alpha}}\ee_k \in \PP{U}} \subset R_A^m (- \dd)$ be a $U$-marked set, where $\Ht(\mf{\underline{\alpha}}{k})=x^{\underline{\alpha}}\, \ee_k$.

\begin{proposition}\label{prop:riscritture}\cite[Proposition 4.11, Lemma 6.1]{ABRS}
With this notation, every element $f$ of $R_A^m (- \dd)$ can be uniquely written as
\begin{equation}\label{eq:every unique writing}
f=\sum_{x^{\underline{\gamma}}\, \mee{\ell}{} \in \PP{U}} P_{\underline{\gamma},\ell} \mf{\underline\gamma}{\ell}+g
\end{equation}
with $g\in \langle \mathcal N(U)\rangle$, $\mf{\underline{\gamma}}{\ell} \in F$, $P_{\underline{\gamma},\ell} \in R_A$ and $\supp(P_{\underline{\gamma},\ell}\cdot x^{\underline{\gamma}}) \subset \mathcal{C}(x^{\underline{\gamma}})$.

In particular, if $F$ is a $U$-marked basis, then every element $f \in (F)$ can be uniquely written as
\begin{equation}\label{eq:unique writing}
f=\sum_{x^{\underline{\gamma}}\, \mee{\ell}{} \in \PP{U}} P_{\underline{\gamma},\ell} \mf{\underline\gamma}{\ell}.
\end{equation}
\end{proposition} 

\begin{theorem}\label{th:caratterizzazione} \cite[Theorem 4.18]{ABRS} A $U$-marked set $F$ is a $U$-marked basis if, and only if, for each element $\mf{\underline{\alpha}}{k} \in F$ and each non-multiplicative variable $x_i$ of $x^{\underline{\alpha}}$, $x_i \mf{\underline{\alpha}}{k}\cridast{F} 0$.
\end{theorem}
 
The next result highlights the relation between a module $M\subset R_A^m (- \dd)$ generated by a $U$-marked basis and its truncation $M_{\geqslant s}$, which consists of all elements of $M$ of degree $\geqslant s$. For $m=1$, this connection was previously stated in \cite[Theorem 3.4]{LR2} assuming further that the quasi-stable module $U$ is a strongly stable ideal, and also in \cite[Proposition 6.7]{macaulay} in terms of marked schemes. 

\begin{proposition}\label{prop:basisTrunc}
Let $M\subset R_A^m (- \dd)$ be a module generated by a $U$-marked basis $F$. 
For each $s$, $U_{\geqslant s}$ is quasi-stable, and $M_{\geqslant s}$ is generated by a $U_{\geqslant s}$-marked basis.
\end{proposition}



\begin{proof} 
If $s$ is less than or equal to the minimal degree of an element in $U$ (or $M$), then $U_{\geqslant s} = U$ and $M_{\geqslant s} = M$. If $s > \min \{\deg (x^{\underline{\alpha}}\, \ee_k)\ \vert\ x^{\underline{\alpha}}\,\ee_k \in \mathcal{P}_U\}$, then
\[
\mathcal{C}(x^{\underline{\alpha}}\ee_k)_{\geqslant s} = \bigsqcup_{x^{\underline {\gamma}}\, \ee_k \in \mathcal{C}(x^{\underline{\alpha}}\, \ee_k)_s} \mathcal C(x^{\underline {\gamma}}\, \ee_k)
\]
Therefore, the monomial module $U_{\geqslant s}$ is quasi-stable with Pommaret basis 
\[
\PP{U_{\geqslant s}} = \bigcup_{x^{\underline{\alpha}}\ee_k \in \PP{U}}  \mathcal{C}(x^{\underline{\alpha}}\, \ee_k)_s.
\]
By Proposition \ref{prop:riscritture}, the truncated module $M_{\geqslant s}$ is generated by
\[
G := F_{\geqslant s} \cup\left(\bigcup_{\begin{subarray}{c}x^{\underline{\alpha}}\, \ee_k \in \PP{U}\\ \deg(x^{\underline{\alpha}}\, \ee_k) < s\end{subarray}} \left\{ x^{\underline{\delta}} f_{\underline{\alpha},k}\ \middle\vert\ x^{\underline{\delta}}x^{\underline{\alpha}}\ee_k \in \mathcal{C}(x^{\underline{\alpha}}\ee_k)_s \right\}\right).
\]
Moreover, define the set of elements
\[
\widetilde{F} := \left\{ \tilde{f}_{\underline{\gamma},k} = \HTemph{x^{\underline{\gamma}}\, \ee_k} - g_{\underline{\gamma},k}\ \middle\vert\ x^{\underline{\gamma}}\, \ee_k \in \mathcal{P}_{U_{\geqslant s}}  \right\}
\]
where $g_{\underline{\gamma},k}$ is the remainder in \eqref{eq:every unique writing} obtained from reducing the term $x^{\underline{\gamma}}\, \ee_k$ with respect to the marked basis $F$.

The set $\widetilde{F}$ is a $U_{\geqslant s}$-marked set; the subsets $F_{> s}$ and $\widetilde{F}_{> s}$ coincide and the homogeneous pieces $G_s$ and $\widetilde{F}_s$ of degree $s$ generate the same $A$-module. Therefore, by Theorem~\ref{th:caratterizzazione}$,\widetilde{F}$ is a $U_{\geqslant s}$-marked basis.
\end{proof}

Theorem \ref{th:caratterizzazione} and the expression \eqref{eq:unique writing} lead to specific syzygies of a marked basis, which we now focus on.

Given a $U$-marked set $F = \{\mf{\underline{\alpha}_j}{k_j}\}_{j=1,\dots,p}\subset R_A^m(-\dd)$, set $d'_j := \deg(\mf{\underline{\alpha}_j}{k_j}) = \vert \underline{\alpha}_j \vert + d_{k_j}$ and $\dd' = (d'_1,\ldots,d'_p)$. Moreover, consider $R_A^p(-\dd') = \bigoplus_{j=1}^p R_A(-d_j')\ff_{\underline{\alpha}_j,k_j}$ with standard basis $\{\ff_{\underline{\alpha}_j,k_j}\}_{j=1,\dots,p}$, and define the module homomorphism $\partial_0: R_A^p(-\dd') \to R_A^m(-\dd)$ by 
 \[
\ff_{\underline{\alpha}_j,k_j} \to \mf{\underline{\alpha}_j}{k_j},\qquad j=1,\ldots,p.
 \]
Since $(F) = \textnormal{im}(\partial_0)$, we denote by $\syz(F) = \ker(\partial_0)\subset R_A^p(-\dd')$ the $R_A$-module of first syzygies of $F$.

\begin{definition}\label{def:fundamental} \cite[Section 6]{ABRS} 
Given an $U$-marked basis $F$, a syzygy of $F$  of the form
\begin{equation}\label{eq:fundamental-syzygy}
\msyz{i,\underline{\alpha},k}{}= \HTemph{x_i\ff_{\underline{\alpha},k}} - \sum_{x^{\underline{\gamma}}\, \mee{\ell}{} \in \PP{U}} P_{\underline{\gamma},\ell}\ff_{\underline{\gamma},\ell},
\end{equation}
is called a {\em fundamental syzygy of $F$} if  it is obtained via the unique writing of $x_i \mf{\underline{\alpha}}{k}$ as in  \eqref{eq:unique writing}, where $\mf{\underline{\alpha}}{k}\in F$ and  $x_i$ is a non-multiplicative variable of $x^{\underline{\alpha}}$. We denote by 
\[
F^{\partial_1}:=\left\{\msyz{i,\underline{\alpha},k}{}\ \middle\vert\ \mf{\underline{\alpha}}{k} \in F,\ x_i > \min(x^{\underline{\alpha}})\right\}\subseteq \syz(F)
\]
the set of all fundamental syzygies of $F$. 
\end{definition} 

\begin{remark}\label{rk:sliced pommaret cone}
According to \cite[Lemma 3.4\emph{(vi)}]{ABRS}, the reduction $\xrightarrow{F^\ast}$ of elements $x_i\mf{\underline{\alpha}}{k}$ involves terms $x^{\underline{\eta}} \mf{\underline{\gamma}}{\ell}$ with $x^{\underline{\eta}} <_{\textnormal{Lex}} x_i$, i.e., $x^{\underline{\eta}} \in A[x_0,\ldots,x_{i-1}]$.  We call {\em $i$-sliced Pommaret cone} of a term $x^{\underline{\alpha}}\,\ee_k$ the following subset of the Pommaret cone $\mathcal{C}(x^{\underline{\alpha}}\ee_k)$
\begin{equation}
\mathcal{C}_i(x^{\underline{\alpha}}\,\ee_k)   := \left\{x^{\underline{\delta}} x^{\underline {\alpha}}\, \ee_k\ \middle\vert\ \max(x^{\underline{\delta}})\leqslant  \min\{x_{i-1},\min(x^{\underline {\alpha}})\} \right\}.
\end{equation}

Hence, the support of $P_{\underline{\gamma},\ell}$ in \eqref{eq:fundamental-syzygy} is contained in the $i$-sliced Pommaret cone $\mathcal{C}_i(x^{\underline{\gamma}})$, and the support of the fundamental syzygy $\msyz{i,\underline{\alpha},k}{}$ is contained in
\begin{equation}\label{supporto1}
\{x_i\,\ff_{\underline{\alpha},k}\} \cup  \bigcup_{x^{\underline{\gamma}}\,\mathbf{e}_\ell \in \mathcal{P}_U} \left\{x^{\underline{\eta}}\,\ff_{\underline{\gamma},\ell}\ \middle\vert\ x^{\underline{\eta}} \in \mathcal{C}_i(x^{\underline{\gamma}})_{d'_{k} +1 - d'_{\ell}}\right\}.
\end{equation}
\end{remark}

\begin{remark}\label{rk:syzygy modules stable}
It is noteworthy that  $F^{\partial_1}$ naturally forms a $U'$-marked set, where $U' \subset R_A^m(-\dd')$ is the \emph{stable} submodule with Pommaret basis 
\begin{equation}\label{eq:pommaretBasisU'}
\mathcal P_{U'}=\left\{x_i\ff_{\underline{\alpha},k}\ \middle\vert\ \ff_{\underline{\alpha},k} \in R_A^m(-\dd'),\ x_i > \min(x^{\underline{\alpha}})\right\}
\end{equation}
In particular, $U'$ is generated linearly
\[
U' = \bigoplus_{x^{\underline{\alpha}}\ee_k} \big(x_i\ \vert\ x_i > \min(x^{\underline{\alpha}})\big)\ff_{\underline{\alpha},k}
\]
\end{remark}

\begin{theorem}\label{th:fundamental} \cite[Lemma 6.3 and Theorem 6.5]{ABRS}
Let $F$ be a $U$-marked basis, and let $U'$ be the  stable module in $R_A^p(-\dd')$ with Pommaret basis $\PP{U'}$ as described in \eqref{eq:pommaretBasisU'}. Then, $F^{\partial_1}$ is the $U'$-marked basis of the module $Syz(F)$ of the first syzygies of $F$.
\end{theorem}

Thanks to Theorem \ref{th:fundamental}, the following refinement of formula \eqref{supporto1} is straightforward and provides a prediction of the shape of the marked basis of the first syzygy module.

\begin{corollary}\label{cor:suppSyz}
With the notation of Theorem \ref{th:fundamental}, the support of  the fundamental syzygy $\msyz{i,\underline{\alpha},k}{}$ is contained in the following set:
\begin{equation}\label{supporto2}
\{x_i\,\ff_{\underline{\alpha},k}\} \cup \left( \bigcup_{x^{\underline{\gamma}}\,\mathbf{e}_\ell \in \mathcal{P}_U} \left\{x^{\underline{\eta}}\,\ff_{\underline{\gamma},\ell}\ \middle\vert\ x^{\underline{\eta}} \in \mathcal{C}_i(x^{\underline{\gamma}})_{d'_{k} +1 - d'_{\ell}}\right\} \cap \mathcal N(U')\right).
\end{equation}
\end{corollary}

We can further refine Corollary \ref{cor:suppSyz} by predicting which terms will not appear in the support of a fundamental syzygy: the goal is to identify terms $x^{\underline \eta}$ in $ \mathcal{C}_i(x^{\underline{\gamma}})$ that never satisfy $x^{\underline \eta}x^{\underline \gamma}\ee_\ell=x_i x^{\underline \epsilon}\ee_\ell$ for any $x^{\underline \epsilon}\ee_\ell\in \mathcal N(U)$.

\begin{corollary}\label{cor:forecastNoTerms}
With the notation of Theorem \ref{th:fundamental} and Corollary \ref{cor:suppSyz}, if $x^{\underline \eta}\in  \mathcal{C}_i(x^{\underline{\gamma}})_{d'_{k} +1 - d'_{\ell}}$ with $x^{\underline{\eta}}\,\ff_{\underline{\gamma},\ell}\in \mathcal N(U')$ and for every $x^{\underline \epsilon} \ee_\ell\in \mathcal N(U)$ we have that $x_ix^{\underline \epsilon} \ee_\ell\in U$ but $x_ix^{\underline \epsilon}\ee_\ell\neq x^{\underline \eta}x^{\underline \gamma}\ee_\ell$, then $x^{\underline \eta} \ff_{\underline{\gamma},\ell}$ does not appear in the support of $\msyz{i,\underline{\alpha},k}{}$.
\end{corollary}

\begin{example}\label{ex:section2-part2}
Consider the ideal $J = (x_2^2,x_1^2) \subset R = \kk[x_0,x_1,x_2]$ introduced in Example \ref{ex:section2-part1} and the $J$-marked set
\[
F = \left\{ f_{x_2^2} = \HTemph{x_2^2} - x_2x_1 - x_2x_0 + x_1 x_0,\  f_{x_1^2} = \HTemph{x_1^2} + a\, x_2x_1,\ f_{x_2x_1^2} = \HTemph{x_2x_1^2} - 2x_2x_1x_0 + 2x_1x_0^2 \right\},
\] 
depending on a parameter $a\in\mathbb{K}$. $F$ is a $J$-marked basis if, and only if, $a = -1$. In fact,
\[
\begin{split}
x_2 f_{x_1^2}=  {}&ax_1\, f_{x_2^2} - ax_0\, f_{x_1^2} + (a+1)\, f_{x_2 x_1^2} + {}\\
&(a^2+4a+3) x_2x_1x_0 - (a+1)x_1x_0^2\\
x_2 f_{x_2x_1^2} = {}& (x_1^2 - 2x_1x_0)\, f_{x_2^2} - x_1x_0\, f_{x_1^2} + \big(x_1 + (a+1)x_0\big)\, f_{x_2 x_1^2} + {}\\
& 2(a+1)x_2x_1x_0 - 2(a+1)x_1x_0^3
\end{split}
\]
Set $a=-1$ and given the homomorphism $\partial_0: R(-2)\ff_{x_2^2} \oplus R(-2)\ff_{x_1^2} \oplus R(-3)\ff_{x_2x_1^2} \to R$ defined by
\[
\ff_{x_2^2} \mapsto f_{x_2^2},\qquad\ff_{x_1^2} \mapsto f_{x_1^2},\qquad \ff_{x_2x_1^2} \mapsto  f_{x_2x_1^2},
\]
the stable module $U'$ is generated by the Pommaret basis $\mathcal{P}_{U'} = \{x_2\ff_{x_1^2} ,x_2\ff_{x_2x_1^2}\}$ and the syzygy module $\textnormal{Syz}(F) = \ker \partial_0$ is generated by the $U'$-marked basis
\[
F^{\partial_1} = \left\{\HTemph{x_2\ff_{x_1^2}} + x_1\ff_{x_2^2} - x_0\ff_{x_1^2},\  \HTemph{x_2\ff_{x_2x_1^2}} + (2x_1x_0-x_1^2)\ff_{x_2^2} + x_1x_0\ff_{x_1^2} - x_1 \ff_{x_2x_1^2}\right\}.
\]
\end{example}

By repeatedly applying Theorem \ref{th:fundamental}, one can construct a free resolution of a graded module $M$ generated by a $U$-marked basis for some quasi-stable module $U$. Consequently, the direct summands in such a resolution depend solely on the Pommaret basis $\PP{U}$.

\begin{theorem}\label{thm:freeRes}
\cite[Theorem 6.6]{ABRS}
Let $A$ be a Noetherian $\kk$-algebra, $U$ a quasi-stable module in $R_A^m (- \dd)$, and $M = (F)$ a module generated by a $U$-marked basis $ F\subset R_A^m (- \dd)$.

Denote by $g_{h,q}(U)$ the number of terms $x^{\underline {\alpha}}\, \ee_k \in
\PP{U}$ with $\deg(x^{\underline {\alpha}}\,\ee_k)=h$ and $\min (x^{\underline {\alpha}}) = x_q$, and
set $D=\min\{i\ \vert\ x_i=\min(x^{\underline {\alpha}}),\ x^{\underline {\alpha}}\, \ee_k \in \PP{U}\}$. Then
$M$ admits a finite free resolution
\begin{equation}\label{eq:freeres}
0\xrightarrow{}\bigoplus R_A(-j)^{r_{n-D,j}}\xrightarrow{\partial_{n-D}}\cdots
\xrightarrow{}\bigoplus R_A(-j)^{r_{1,j}}
\xrightarrow{\partial_1}\bigoplus R_A(-j)^{r_{0,j}}\xrightarrow{\partial_0} M \xrightarrow{} 0
\end{equation}
of length $n - D$, where the ranks of the free modules are determined by
\begin{equation}\label{eq:rksBetti}
r_{i,j} := r_{i,j}(U) = \sum_{q = 1}^{n-i} \binom{n-q}{i}g_{j-i,q}(U).
\end{equation}
\end{theorem}

From now on, we refer to the resolution \eqref{eq:freeres} as the \emph{$U$-resolution} of $M$, and the numbers $r_{i,j}$ in \eqref{eq:rksBetti} as the Betti numbers of the $U$-resolution. Moreover, we denote by $\beta_{i,j}(M)$ the graded Betti numbers of $M$, i.e., the Betti numbers of its minimal resolution. We also denote by $\reg(M)$ and $\mathrm{pdim}(M)$ its regularity and projective dimension of $M$, respectively.

\begin{corollary}\label{cor:maggiorazioni}
With the same notation as in Theorem \ref{thm:freeRes},
\begin{enumerate}[(i)]
\item $\beta_{i,j}(M)\leqslant r_{i,j}(U)$
\item $\reg(M)\leqslant \reg(U)$
\item $\mathrm{pdim}(M)\leqslant \mathrm{pdim}(U)$.
\end{enumerate}
Furthermore, for every $t>0$, let $U^{(t)}$ be the stable module such that the $t$-th fundamental syzygies form a $U^{(t)}$-marked basis of the $t$-th syzygy module $\ker(\partial_{t-1})$. Then,
\begin{enumerate}[(i)]\setcounter{enumi}{3}
\item
$\beta_{i,j}(M) \leqslant \beta_{i-t,j}(U^{(t)}) \leqslant r_{i,j}(U), \qquad\forall\ i \geqslant  t,\ \forall\ j$.
\end{enumerate}
\end{corollary}

\begin{proof}
For the first three items, see \cite[Corollary 6.8]{ABRS}. For the last item, let $\textnormal{Syz}^{(t)}(M) = \ker(\partial_{t-1})$ be the $t$-th syzygy module of $M$, which has a $U^{(t)}$-marked basis. Then, one has
\[
\begin{tikzpicture}
\node at (0,0) {$0\xrightarrow{}\bigoplus R_A(-j)^{r_{n-D,j}}\xrightarrow{\partial_{n-D}}\cdots
\xrightarrow{}\bigoplus R_A(-j)^{r_{t,j}}
\xrightarrow{\partial_t}\bigoplus R_A(-j)^{r_{t-1,j}}\xrightarrow{\partial_{t-1}} \cdots \xrightarrow{\partial_0} M \xrightarrow{} 0$};
\begin{scope}[shift={(-1,0)}]
\node (ker) at (2,-1) [inner sep=2pt] {$\ker(\partial_{t-1})$};
\draw[->] (ker) -- (2.8,-0.35);
\draw[->] (0.8,-0.35) -- (ker);
\node (01) at (3.2,-1.65) [inner sep=2pt] {$0$};
\draw[->] (ker) -- (01);
\node (02) at (0.8,-1.65) [inner sep=2pt] {$0$};
\draw[<-] (ker) -- (02);
\end{scope}
\end{tikzpicture}
\]
so that the $U$-resolution of $M$ contains the $U^{(t)}$-resolution of $\textnormal{Syz}^{(t)}(M)$. Using the first part of the corollary, we obtain
\[
\beta_{i,j}(M) = \beta_{i-t,j}(\textnormal{Syz}^{(t)}(M)) \leqslant \beta_{i-t,j}(U^{(t)}) \leqslant r_{i-t,j}(U^{(t)}) \leqslant r_{i,j}(U),\quad \forall\ i \geqslant t.\qedhere
\]
\end{proof}


\section{Minimal $U$-resolutions}\label{sec:minimal-resolution}

Starting from Theorem \ref{thm:freeRes} and Corollary \ref{cor:maggiorazioni}, it is natural to ask whether the Betti numbers $r_{i,j}$ of the $U$-resolution \eqref{eq:freeres} of the module $M$ actually coincide with its own Betti numbers, i.e., whether a $U$-resolution is minimal. 

Recall that a free resolution is minimal if and only if every entry of its matrices is non-invertible. In the case of a $U$-resolution, the entries of the matrices belong to $R_A$, whose invertible elements are the invertible elements of $A$. An invertible element of $A$ is a non-null element, but a non-zero element of $A$ can be not invertible, unless we assume that $A$ is a field.

The first step toward answering this question is to analyze the behavior of the resolution for the monomial module $U$ itself.

\begin{lemma}\label{lemma:minstable}
Let $U \subset R_A^m(-\dd)$ be a quasi-stable module. Then, the $U$-resolution of $U$ is minimal if, and only if, $U$ is stable. Consequently, for a stable module $U$, $\beta_{i,j}(U)=r_{i,j}$.  
\end{lemma}

\begin{proof}
The claim follows from \cite[Theorem 8.9]{Seiler2009II}, where exactly the same resolution is discussed. However, for the sake of completeness we provide a brief proof. 

If $U =\bigoplus J^{(k)}\ee_k$ is stable, then each ideal $J^{(k)} \subset R_A$ is stable, and the resolution \eqref{eq:freeres} coincides to the Eliahou-Kervaire resolution, which is minimal \cite{EK}.

If $U$ is quasi-stable but not stable, then $\mathcal{B}_U \subsetneq \mathcal{P}_U$. More precisely, there exists a term $x^{\underline{\gamma}} \,\ee_k \in \mathcal{P}_U \setminus \mathcal{B}_U$, such that
\[
x^{\underline{\gamma}} \,\ee_k = x_i \, x^{\underline{\alpha}} \,\ee_k \quad\text{with}\quad x^{\underline{\alpha}} \,\ee_k \in \mathcal{B}_U\text{~and~}x_i > \min x^{\underline{\alpha}}.
\]
Hence, among the set of fundamental syzygies $F^{\partial_1}$ we find $x_i\ff_{\underline{\alpha},k} - \ff_{\underline{\gamma},k}$, and the matrix of the morphism $\partial_1$ in \eqref{eq:freeres} contains an invertible element of $A$.
\end{proof}

The second step is to understand what happens for the modules belonging to $\MFFunctor{U}(A)$, that is, for the modules $M$ generated by a $U$-marked basis. If $U$ is stable, is the $U$-resolution of every module $M$ minimal? If $U$ is quasi-stable but not stable, is the $U$-resolution of every module $M$ non-minimal? The following examples show that the minimality (resp.~non-minimality) of the resolution of $U$ does not imply the minimality (resp.~non-minimality) of the $U$-resolution of $M$.
In particular, in Example~\ref{ex:minimality-quasi-stable} we present an explicit module $M$ generated by a $U$-marked basis such that $\beta_{i,j}(M)=r_{ij}> \beta_{i,j}(U)$ for some~$i, j$.

\begin{example}\label{ex:minimality-quasi-stable}
Consider again the quasi-stable ideal $J = (x_2^2,x_1^2) \subset R = \kk[x_0,x_1,x_2]$, and let $I$ be the ideal generated by the $J$-marked basis introduced in Example \ref{ex:section2-part2}. Even though the $J$-resolution of $J$ is not minimal
\[
0 \xrightarrow{} R(-3)\oplus R(-4)\xrightarrow{\scriptsize\left[\begin{array}{cc} 0 & -x_1^2 \\ \HTemph{x_2} & 0 \\ -1 & \HTemph{x_2} \end{array}\right]}R(-2)^2 \oplus R(-3) \xrightarrow{\scriptsize\left[\begin{array}{ccc} \HTemph{x_2^2} & \HTemph{x_1^2} & \HTemph{x_2x_1^2}\end{array}\right]} J \xrightarrow{} 0,
\]
the $J$-resolution of the ideal $I$ is, in fact, minimal
\[
0 \xrightarrow{} R(-3) \oplus R(-4)\xrightarrow{\scriptsize\left[\begin{array}{cc} x_1 & -x_1^2+2x_1x_0 \\ \HTemph{x_2}-x_0 & x_1x_0 \\ 0 & \HTemph{x_2} - x_1 \end{array}\right]}R(-2)^2 \oplus R(-3)\xrightarrow{\scriptsize\left[\begin{array}{ccc} f_{x_2^2} & f_{x_1^2} & f_{x_2x_1^2}\end{array}\right]} I \xrightarrow{} 0.
\]
On the other hand, let $K$ be the ideal generated by the following $J$-marked basis
\[
 G= \left\{ g_{x_2^2} = \HTemph{x_2^2} - \tfrac{1}{2}x_2x_1- x_2x_0 ,\  g_{x_1^2} = \HTemph{x_1^2} -\tfrac{1}{2} x_2x_1 - x_1x_0,\ g_{x_2x_1^2} = \HTemph{x_2x_1^2} - 2x_2x_1x_0 \right\}.
\]
This ideal has a non-minimal $J$-resolution
\[
0 \xrightarrow{} R(-3) \oplus R(-4)\xrightarrow{\scriptsize\left[\begin{array}{cc} \tfrac{1}{2}x_1 & -x_1^2+2x_1x_0 \\ \HTemph{x_2} & 0 \\ -\tfrac{3}{4} & \HTemph{x_2} - \tfrac{1}{2}x_1-x_0 \end{array}\right]}R(-2)^2 \oplus R(-3) \xrightarrow{\scriptsize\left[\begin{array}{ccc} g_{x_2^2} & g_{x_1^2} & g_{x_2x_1^2}\end{array}\right]} K \xrightarrow{} 0.
\]
\end{example}

\begin{example}\label{ex:minimality-stable}
The ideals $I$ and $K$ considered in the previous examples can be generated also by 
$J'$-marked bases, where  $J'$ is the stable ideal $(x_2^2,x_2x_1,x_1^3) \subset \kk[x_0,x_1,x_2]$. The ideal $I$ (resp.~$K$) is generated by $F'$ (resp.~$G'$):
\[
\begin{split}
F' &= \left\{ f'_{x_2^2} = \HTemph{x_2^2}-x_1^2-x_2x_0+x_1x_0,\ f'_{x_2x_1}= \HTemph{x_2x_1}-x_1^2,\ f'_{x_1^3} = \HTemph{x_1^3}-2x_1^2x_0+2x_1x_0^2\right\},\\
G' &= \left\{ g'_{x_2^2} = \HTemph{x_2^2}-x_1^2-x_2x_0 + x_1x_0,\ g'_{x_2x_1}= \HTemph{x_2x_1}-2x_1^2+2x_1x_0,\ g'_{x_1^3} = \HTemph{x_1^3}-3x_1^2x_0+2x_1x_0^2\right\}.
\end{split}
\]
In this case, the $J'$-resolution is minimal for both $J'$ and $I$, but non-minimal for $K$, since the ranks $r_{i,j}$ are the same in both the $J$- and $J'$-resolutions.
\end{example}

Examples \ref{ex:minimality-quasi-stable} and \ref{ex:minimality-stable} show that, by deforming the Pommaret basis of a quasi-stable module $U$ into a $U$-marked basis, the presence or absence of an invertible element in one of the matrices of the $U$-resolution can change. 
Under some simple sufficient conditions, the following more precise results hold for a stable module $U$ .

Following \cite{HH1999, Seiler2009II}, recall that a free resolution of a  graded module $M$ is said to be \emph{linear} if the non-zero entries of its matrices are linear forms. It is straightforward that a linear free resolution is minimal, and that if $M$ admits a linear free resolution, then $M$ is generated in a single degree. 

\begin{proposition}\label{prop:StableSingleDegree}
Let $U \subset R_A^m$ be a stable module, and let $M\subset R_A^m$ be any module generated by a $U$-marked basis.  
\begin{enumerate}[\it (i)]
\item If $M$ is generated in a single degree $s>0$, i.e., $M=(M_s)$, then the $U$-resolution of $M$ is linear.

\item Assume that $A=\mathbb K$. If $U$ has no minimal generators in two consecutive degrees, then the $U$-resolution of $M$ is minimal.
\end{enumerate}
In both  cases (i) and (ii), $\reg(M)=\reg(U)$.
\end{proposition}

\begin{proof}
\emph{(i)} Since $U$ is generated only in one degree $s>0$, the $U$-resolution of $M$ takes the form
\[
0\xrightarrow{} R_A(-s-n+D)^{r_{n-D,s+n-D}}\xrightarrow{\partial_{n-D}}\cdots
\xrightarrow{} R_A(-s-1)^{r_{1,s+1}}
\xrightarrow{\partial_1} R_A(-s)^{r_{0,s}}\xrightarrow{\partial_0} I \xrightarrow{} 0. 
\]

\emph{(ii)} If the free module encoding the generators in the $U$-resolution of $M$ is
\[
R_A(-j_0)^{r_{0,j_0}} \oplus R_A(-j_1)^{r_{0,j_1}} \oplus \cdots \oplus R_A(-j_p)^{r_{0,j_p}}
\]
then the free module at step $t$ of the resolution has the form
\[
R_A(-j_0-t)^{r_{t,j_0+t}} \oplus R_A(-j_1-t)^{r_{t,j_1+t}} \oplus \cdots \oplus R_A(-j_p-t)^{r_{0,j_p+t}}.
\]
The assumption $j_i - j_{i-1} > 1$ for $i = 1,\ldots,p$ ensures that the same degree cannot appear in two consecutive free modules in any $U$-resolution, so that there cannot be non-zero constant entries in the matrices of the morphisms. 
\end{proof}

One may now wonder whether, in the case of a quasi-stable but not stable module $U$, the Betti numbers of $U$ provide a lower bound for the Betti numbers of a module $M$ generated by a $U$-marked basis. The following examples show that this is not necessarily the case.

\begin{example}
Consider the ideals introduced in Examples \ref{ex:minimality-quasi-stable} and \ref{ex:minimality-stable}, and define $U = J\ee_1 \oplus J'\ee_2$,  $M = I\ee_1 \oplus I\ee_2$, and $N = K\ee_1 \oplus K\ee_2$. The modules $M$ and $N$ have $U$-marked bases $F\ee_1 \cup F'\ee_2$ and $G\ee_1 \cup G'\ee_2$, respectively, their Betti tables are
\[
\begin{array}{r|cc}r_{i,j} & 0 & 1 \\ \hline 2 & 4 & 2  \\ 3 & 2 & 1\end{array}\qquad\qquad\qquad \begin{array}{r|cc}U & 0 & 1 \\ \hline 2 & 4 & 1  \\ 3 & 1 & 1\end{array}  \qquad\quad\begin{array}{r|cc}M & 0 & 1 \\ \hline 2 & 4 & 2  \\ 3 & 2 & 1\end{array}\qquad\quad \begin{array}{r|cc}N & 0 & 1 \\ \hline 2 & 4 & 0  \\ 3 & 0 & 1\end{array}
\]
and the $U$-resolution is minimal only for the module $M$.
\end{example}

\begin{example}
By combining contributions of both stable and quasi-stable but not stable ideals, one can also construct quasi-stable ideals whose deformations, defined by marked bases, exhibit both increasing and decreasing Betti numbers. Starting again from the ideals $J$ and $J'$ introduced in Example \ref{ex:section2-part1}, consider the quasi-stable ideal $U = x_2^2 J + x_1^2 J' \subset \kk[x_0,x_1,x_2]$, with Pommaret basis $x_2^2\mathcal{P}_J \cup x_1^2 \mathcal{P}_{J'}$, and the ideals $I$ and $K$ generated by the $U$-marked bases
\[
\begin{split}
&F_I = \{\HTemph{x_2^4},\hspace{1.1cm}\HTemph{x_2^2 x_1^2},\hspace{0.9cm}\HTemph{x_2x_1^3} + x_1^4,\hspace{0.5cm}\HTemph{x_2^3x_1^2},\ \HTemph{x_1^5} \},\\
&F_{K} =  \{\HTemph{x_2^4} - x_1^4,\ \HTemph{x_2^2 x_1^2}+x_1^4,\ \HTemph{x_2x_1^3} + x_2^3x_1,\ \HTemph{x_2^3x_1^2},\ \HTemph{x_1^5} \}.
\end{split}
\]
The Betti tables are
\[
\begin{array}{r|cc}r_{i,j} & 0 & 1 \\ \hline 4 & 3 & 2  \\ 5 & 2 & 2\end{array}\qquad\qquad\qquad\begin{array}{r|cc}U & 0 & 1 \\ \hline 4 & 3 & 1  \\ 5 & 1 & 2\end{array}\qquad\qquad \begin{array}{r|cc}I & 0 & 1 \\ \hline 4 & 3 & 0  \\ 5 & 0 & 2\end{array}\qquad\qquad  \begin{array}{r|cc}K & 0 & 1 \\ \hline 4 & 3 & 2  \\ 5 & 2 & 2\end{array}
\]
and the $U$-resolution is minimal only for the ideal $K$.
\end{example}

Our next goal is to identify a general condition ensuring that the $U$-resolution of a module $M$ generated by a $U$-marked basis is minimal.

The following statement and its consequences extend to marked bases some results that are already known in the case where $J$ is the initial ideal of the ideal $I$, under additional assumptions (see~\cite[Lemma~8.1]{Seiler2009II} and \cite[Corollaries~2.4 and 2.7]{CHH}, respectively). The key point is that the $J$-resolution \eqref{eq:freeres} is constructed using fundamental syzygies. 

\begin{theorem}\label{prop:firstSyzForMin} Let $U$ be a quasi-stable module in $R_A(-\dd)$, and let $M \subset R_A(-\dd)$ be a module generated by a $U$-marked basis $F$.  If the matrix representing the module homomorphism $\partial_{t}$, for $t > 0$, in the $U$-resolution of $M$ contains no non-zero constant entries, then the same holds for the matrix representing $\partial_{t+1}$.
\end{theorem}

\begin{proof}
The columns of the matrix representing $\partial_t: \bigoplus_j R_A(-j)^{r_{t,j}} \to \bigoplus_j R_A(-j)^{r_{t-1,j}}$ describe the $U^{(t)}$-marked basis $F^{\partial_t}$ generating the module of the $t$-th syzygies of $M$. Recall that for every $t>0$, the module $U^{(t)}$ is stable and linearly generated (see Remark \ref{rk:syzygy modules stable}), that is,
\[
U^{(t)} = \bigoplus J^{(t,k)} \ee_k,\qquad \text{where~}J^{(t,k)} = (x_n,\ldots,x_{i_{t,k}}).
\]
By hypothesis, the matrix $\partial_t$ has no non-zero constant entries. Therefore, the elements in $F^{\partial_t}$ are of the form
\[
f_{x_i,k} = \HTemph{x_i\, \ee_k} - \sum_{x^{\underline{\gamma}}\, \ee_\ell \in \mathcal{N}(U^{(t)})} c_{x_i,k}^{\underline{\gamma},\ell}\, x^{\underline{\gamma}}\, \ee_\ell,
\]
where $\deg(x_i\,\ee_k) = \deg(x^{\underline{\gamma}}\, \ee_\ell)$ and $\deg(x^{\underline{\gamma}}) > 0$ (with $c_{x_i,k}^{\underline{\gamma},\ell}$ possibly zero). For each $f_{x_i,k}$ and for any $x_h > x_i$, we construct the fundamental syzygy $\msyz{h,x_i,k}{} \in F^{\partial_{t+1}}$ via the unique writing \eqref{eq:unique writing} 
\[
x_h f_{x_i,k} = \sum_{{x_l}\, \ee_\ell \in \mathcal{P}_{U^{(t)}}}P_{x_l,\ell} f_{x_l,\ell} \qquad \Rightarrow\qquad \msyz{h,x_i,k}{}  =\HTemph{x_h \ff_{x_i,k}} - \sum_{{x_l}\, \ee_\ell \in \mathcal{P}_{U^{(t)}}}P_{x_l,\ell} \ff_{x_l,\ell}.
\]
This writing comes from the reduction $x_h f_{x_i,k} \cridast{F} 0$. Observe that for each term $x_h x^{\underline{\gamma}} \ee_{\ell} \in \supp(x_h f_{x_i,k}) \cap U^{(t)}$, the degree of $x_h x^{\underline{\gamma}}$ is strictly greater than $1$. Hence, in the reduction step $x_h f_{x_i,k} \to x_h f_{x_i,k} - c_{x_i,k}^{\underline{\gamma},\ell} \tfrac{x_h x^{\underline{\gamma}}}{x_l} f_{x_l,\ell}$, where $x_l = \min x^{\underline{\gamma}}$, each $f_{x_l,\ell}$ is always multiplied by a monomial of positive degree.  Repeating this process shows that the polynomials $P_{x_l,\ell}$ in $\msyz{h,x_i,k}{}$ do not contain any constants. Therefore, the matrix representing $\partial_{t+1}$ contains no non-zero constant entries, completing the proof.
\end{proof}

\begin{corollary}\label{cor:minimality-betti-numbers} 
Assume $A=\kk$. 
Using the same notation as in Theorem \ref{prop:firstSyzForMin}:
\begin{enumerate}[\it (i)]
\item the $U$-resolution of $M$ is a minimal resolution if, and only if, the $U$-marked basis $F$ is a minimal set of generators of $M$;
\item if $\beta_{k,j}(M)=r_{k,j}(U)$ for some $k$ and for every $j$, then $\beta_{i,j}(M)=r_{i,j}(U)$ for every $i\geqslant k$ and every $j$. In particular, if $U$ is stable, then
\[
\beta_{k,j}(M)=\beta_{k,j}(U),\ \forall\ j \qquad \Longrightarrow \qquad \beta_{i,j}(M)=\beta_{i,j}(U),\ \forall\ i \geqslant k,\ \forall\ j
\]
\end{enumerate}
\end{corollary}

\begin{proof}
\emph{(i)} One direction is immediate. Conversely, if $F$ is a minimal set of generators of $M$, then the matrix representing $\partial_1$ contains no non-zero constant entries. The statement then follows directly from Theorem \ref{prop:firstSyzForMin}, under the assumption $A=\kk$.

\emph{(ii)} If $\beta_{k,j}(M)=r_{k,j}(U)$ for some $k$ and all $j$, then the matrix representing $\partial_{k+1}$ contains no non-zero constant entries. By Theorem \ref{prop:firstSyzForMin}, the same property holds for all subsequent maps $\partial_i$, and the equality of Betti numbers follows for all $i \geqslant k$.
\end{proof}

\begin{proposition}\label{prop:term order} Let $A=\kk$. 
Consider  a quasi-stable but not stable ideal $J$, and let $I$  be an ideal generated by a $J$-marked basis such that its $J$-resolution is minimal. Then, for any term order $\prec$, the $J$-marked basis $F$ of $I$ is not a Gr\"obner basis of $I$ with initial ideal $J$.
\end{proposition}

\begin{proof} 
Since $J$ is quasi-stable but not stable, the matrix of the homomorphism $\partial_1^J$ in the $J$-resolution of $J$ contains some invertible entries. On the other hand, in the matrix $\partial_1^I$ of the $J$-resolution of $I$, every entry is either zero or non-constant because the resolution of $I$ is minimal.  By the properties of Gr\"obner bases, this situation cannot occur if $F$ were a Gr\"obner basis of $I$ with initial ideal~$J$.
\end{proof}

We conclude this section by focusing on $U$-resolutions of componentwise linear graded modules. Recall that a graded module $M$ is called \emph{componentwise linear} if, for every degree $d\geqslant 0$, the module generated by the component $M_d$ of degree $d$ has a linear resolution; that is, $\reg(M_d)=d$. Equivalently, the only non-vanishing Betti numbers of the module generated by $M_d$ are $\beta_{i,i+d}$ for $i=0,1,\ldots$ (see \cite{HH1999,Seiler2009II}). Hence, componentwise linear modules generalize the situation already considered in item~{\em (i)} of Proposition \ref{prop:StableSingleDegree}, taking into account that a stable monomial module is componentwise linear.

The notion of componentwise linear module, introduced in \cite{HH1999}, plays an important role in the general study of minimal free resolutions. Several related results have been established in terms of initial ideals and generic initial ideals (see \cite{HH1999,Conca,CHH,Seiler2009II}, \cite[Theorem 2.9]{CS2015}, \cite[subsection~3.3]{HSS2018}). As emphasized in \cite{HSS2018}, a generic position associated with stability can be defined in terms of the componentwise linearity, making it natural to develop analogous results in the context of marked bases.

For simplicity, all statements, arguments and proofs in the remainder of this section are presented for ideals in $R_A$, although they can naturally be extended to modules in $R_A^m(-\dd)$. Indeed, as the examples illustrate, each ideal $J^{(k)}$ in the decomposition $U = \bigoplus J^{(k)}\ee_k$ gives rise to an independent subset $F^{\partial_1}$ of the fundamental syzygies. 

The following statement generalizes \cite[Theorem 8.2]{Seiler2009II} and \cite[Theorem~1.1]{AHH} to marked bases in arbitrary characteristic.

\begin{theorem}\label{thmn:minimal-compwLinear}
Let us assume $A=\kk$.
Let $J \subset R$ be a quasi-stable ideal, and let $I\subset R$ be an ideal generated by a $J$-marked basis $F$. If the $J$-resolution of $I$ is minimal, then $I$ is componentwise linear.
\end{theorem}

\begin{proof}
For a given degree $s$, let $\tilde F$ be the $J_{\geqslant s}$-marked basis of $I_{\geqslant s}$, according to Proposition \ref{prop:basisTrunc}, and let $\tilde F_s=\{\tilde f_{\underline{\alpha}}\}$ be the subset of $\tilde F$ consisting of elements of degree $s$ of $\tilde F$. Clearly, $\tilde F_s$ forms linear basis for the $\kk$-module $(I_s)$. 

Since the $J$-resolution of $I$ minimal, the $J$-marked basis $F$ of $I$ is a minimal set of generators by item {\em (i)} of Corollary \ref{cor:minimality-betti-numbers}. Consequently, $\tilde F$ is a minimal set of generators of $I_{\geqslant s}$ as well, and the $J_{\geqslant s}$-resolution of $I_{\geqslant s}$ is minimal. 

Let $x_i$ be a non-multiplicative variable for the head term $x^{\underline \alpha}$ of $\tilde f_{\underline{\alpha}}\in \tilde F_s$, and consider the polynomial $x_i \tilde f_{\underline{\alpha}}$. By Proposition \ref{prop:riscritture}, there exists a unique expression 
\[
x_i \tilde f_{\underline{\alpha}}=\sum_{x^{\underline{\gamma}} \in \PP{U_{\geqslant s}}} P_{\underline{\gamma}} f_{\underline\gamma}
\]
where $P_{\underline{\gamma}}$ belongs to $\mathbb K$ if and only if $f_{\underline\gamma} \in \tilde F\setminus \tilde F_s$. If some $P_{\underline{\gamma}} \neq 0$, then $f_{\underline\gamma}$ would be dependent on other polynomials of $\tilde F$, contradicting the minimality of the $J_{\geqslant s}$-resolution of $I_{\geqslant s}$. Therefore, we have 
\begin{equation}\label{eq:grado s}
x_i \tilde f_{\underline{\alpha}}=\sum_{x^{\underline{\gamma}} \in \PP{U_{s}}} P_{\underline{\gamma}} \tilde f_{\underline\gamma}
\end{equation} 
expressed solely in terms of polynomials of $\tilde F_s$. Formula \eqref{eq:grado s} implies that, in the fundamental syzygy constructed from $x_i \tilde f_{\underline{\alpha}}$, all components corresponding to polynomials in $\tilde F\setminus \tilde F_s$ are zero.

A syzygy $S$ of $\tilde F_s$ is also a syzygy of $\tilde F$,  so it can be rewritten in terms of fundamental syzygies of $\tilde F$. At the first step of the rewriting, one must use a fundamental syzygy of $\tilde F_s$ to rewrite a term corresponding to a polynomial in $\tilde F_s$. By \eqref{eq:grado s}, this action does not introduce non-zero components corresponding to polynomials of $\tilde F\setminus \tilde F_s$. The subsequent steps of the rewriting follow the same behavior.

Hence, at every step, the components corresponding to polynomials in $\tilde F\setminus \tilde F_s$ remain zero. At the end of the rewriting, only fundamental syzygies of $\tilde F_s$ appear with non-zero coefficients. 

In conclusion, any syzygy of $\tilde F_s$ depends solely on the fundamental syzygies of $\tilde F_s$, which have degree $s+1$. By Theorem 2.11, the same argument applies to every subsequent module in the $J_{\geqslant s}$-resolution of $I_{\geqslant s}$, proving that $I$ is componentwise linear.
\end{proof}

\begin{remark}
If $J$ is stable, the proof of Theorem \ref{thmn:minimal-compwLinear} can be approached in an alternative way. In fact, if $J$ is stable, then each $(J_s)$ is also stable. Thanks to formula \eqref{eq:grado s}, $\tilde F_s$ is therefore a $(J_s)$-marked basis by Theorem~\ref{th:caratterizzazione}. We can then conclude by applying Proposition \ref{prop:StableSingleDegree}{\it (i)}, observing that the fundamental syzygies of every subsequent module in the $J_{\geqslant s}$-resolution of $I_{\geqslant s}$  are also marked on stable modules, as noted before formula \eqref{eq:pommaretBasisU'}.
\end{remark}

\begin{example}\label{ex:componentwise}
Returning to Example \ref{ex:minimality-quasi-stable}, we observe that the ideal $I$ is componentwise linear, whereas the ideal $K$ is not. The minimal free resolution of $K$ is indeed:
\[
0 \xrightarrow{} \begin{array}{c} R_A(-4)\end{array}  \xrightarrow{\scriptsize\left[\begin{array}{c} 
2x_0x_1 - 2x_1^2 + x_1x_2\\ 
{-2}x_0x_2 - x_1x_2 + 2x_2^2
\end{array}\right]}\begin{array}{c}R_A(-2)^2\\ \end{array} \xrightarrow{\scriptsize\left[\begin{array}{cc} f'_{x_2^2} & f'_{x_1^2}\end{array}\right]} K\xrightarrow{} 0.
\]
\end{example} 

It is well-known that the converse of Theorem \ref{thmn:minimal-compwLinear} does not hold (see \cite[Example 8.3]{Seiler2009II}). Nevertheless, although Example \ref{ex:componentwise} also shows that there exist componentwise linear ideals with minimal $J$-resolution where $J$ is quasi-stable but not stable, the following result holds.

\begin{theorem} \label{thmn:compwLinear-minimal}
Let us assume $A=\kk$.
Let $I$ be a componentwise linear ideal in $R_A$. Then there exists an open subset of linear change of coordinates $g$ such that $g(I)$ is generated by a $J$-marked basis, for some stable ideal $J$, and the $J$-resolution of $g(I)$ is minimal.
\end{theorem}

\begin{proof}
By \cite[Lemma 1.4]{CHH}, the generic initial ideal of a componentwise linear ideal with respect to the reverse lexicographical order induced by $x_0<x_1<\dots<x_n$ is stable, independently of the characteristic of the field. Hence, there exists an open subset of linear changes of coordinates $g$ such that $g(I)$ is generated by a reduced Gr\"obner basis $G$ with initial ideal equal to a stable ideal $J$. We conclude either by \cite[Theorem 9.12]{Seiler2009II}, or, in characteristic zero, by \cite[Theorem~1.1]{AHH}, noting that $G$ is also a $J$-marked basis.
\end{proof}

\begin{remark}\label{rem: oss 1}
Combining the results of Proposition \ref{prop:term order} and  Theorems \ref{thmn:minimal-compwLinear} and \ref{thmn:compwLinear-minimal}, we summarize the following: there exist componentwise linear ideals with marked bases over quasi-stable ideals that are not Gr\"obner bases. Nevertheless, up to a change of variables, they also admit marked bases over stable ideals, which are Gr\"obner bases.
\end{remark}


\section{The syzygy functors}\label{sec:SyzFunct}

In this section, we first recall the definition of the $U$-marked functor, introduced in \cite[Section~5]{ABRS}, which generalizes results from \cite{CMR2015,LR2}. We then highlight an immediate application of this definition to the study of the minimality of a $U$-resolution. Finally, we define functors which parametrize all (or part of) the $U$-resolutions and analyze their relationships.

\begin{definition}[{\cite[Section 5]{ABRS}}]
Ler $U=\bigoplus_{k=1}^m J^{(k)} \ee_k \subseteq R^m(- \dd)$ be a quasi-stable module. We define the {\em $U$-marked functor} as the covariant functor $\MFFunctor{U}: \underline{\text{Noeth~$\kk$-Alg}} \rightarrow \underline{\text{Sets}}$
that associates to any Noetherian $\kk$-algebra $A$ the set
\begin{equation}\label{eq:marked-functor}
\MFFunctor{U}(A) :=  \left\{ M \subset R_A^m (- \dd)\ \middle\vert\ M \oplus \langle\mathcal{N}(U)\rangle=  R_A^m (- \dd)\right\}
\end{equation}
and to any morphism $\sigma: A \rightarrow {A'}$ the map
\[
\begin{array}{rccc}
\MFFunctor{U}(\sigma):& \MFFunctor{U}(A) &\longrightarrow& \MFFunctor{U}({A'})\\
&M & \longmapsto& \sigma(M)
\end{array}
\]
where $\sigma$ also denotes its natural extension to the polynomial rings over $A$ and $A'$. Note that the image $\sigma(F)$ of the marked basis $F$ generating $M$ is again a $U$-marked basis, thanks to the monicity of head terms.
\end{definition}

This functor is representable, and the representing scheme $\MFScheme{U}$ parametrizes all the $U$-marked bases. Indeed, the condition $M \oplus \langle\mathcal{N}(U)\rangle=  R_A^m (- \dd)$ in \eqref{eq:marked-functor} is equivalent to requiring that $M$ is generated by a $U$-marked basis (see \cite[Theorem 4.13, Corollary 4.17, Theorem 5.1]{ABRS}). 
We briefly recall how to determine the equations defining the scheme $\MFScheme{U}$. For every term $x^{\underline{\alpha}}\mee{k}{}\in \PP{U}$, consider the element 
\begin{equation}\label{eq:marked set}
F_{\underline{\alpha},k} := \HTemph{x^{\underline{\alpha}}\mee{k}{}} - \sum_{\begin{subarray}{c}x^{\underline{\beta}}\mee{\ell}{} \in \mathcal{N}(U) \\ \deg x^{\underline{\beta}}\mee{\ell}{} = \deg x^{\underline{\alpha}}\mee{k}{} \end{subarray}} C_{\underline{\alpha},k}^{\underline{\beta},\ell}\, x^{\underline{\beta}}\mee{\ell}{}
\end{equation}
in the module $R^m_{\kk[C]}(-\dd)$, where $\kk[C]$ is the polynomial ring in the parameters
\[
C:=\left\{\pC{\underline\alpha}{k}{\underline\beta}{\ell} \ \middle\vert\ x^{\underline {\alpha}}\ee_k \in \PP{U}, x^{\underline{\beta}} \mee{\ell}{} \in \mathcal N(U), \deg(x^{\underline{\beta}} \mee{\ell}{})=\deg(x^{\underline {\alpha}}\ee_k)\right\}.
\]
Let $\mathcal F$ denote the set of all elements \eqref{eq:marked set}. By construction, $\mathcal F$ is a $U$-marked set. Hence, for each polynomial $F_{\underline{\alpha},k}$ and for every non-multiplicative variable $x_i$ of $x^{\underline\alpha}$, we can compute the $\mathcal F$-normal form of $x_i F_{\underline{\alpha},k}$:
\[
x_i F_{\underline{\alpha},k} \xrightarrow{\mathcal F^{\ast}}_\star \sum_{\begin{subarray}{c}x^{\underline{\eta}}\mee{\ell}{} \in \mathcal{N}(U) \\ \deg x^{\underline{\eta}}\mee{\ell}{} = \deg x^{\underline{\alpha}}\mee{k}{}+1 \end{subarray}} Q_{i,\underline{\alpha},k}^{\underline\eta,\ell}(C)\, x^{\underline\eta}\mee{\ell}{}
\]
This normal form arises from the unique writing (Proposition \ref{prop:riscritture})
\begin{equation}\label{eq:unique}
x_i F_{\underline{\alpha},k} = \sum_{x^{\underline{\gamma}}\, \mee{\ell}{} \in \PP{U}} P^{\underline{\gamma},\ell}_{i,\underline{\alpha},k}(C) F_{\underline\gamma,\ell} + 
\sum_{\begin{subarray}{c}x^{\underline{\eta}}\mee{\ell}{} \in \mathcal{N}(U) \\ \deg x^{\underline{\eta}}\mee{\ell}{} = \deg x^{\underline{\alpha}}\mee{k}{}+1 \end{subarray}}
Q_{i,\underline{\alpha},k}^{\underline\eta,\ell}(C)\, x^{\underline\eta}\mee{\ell}{}.
\end{equation}
By Theorem \ref{th:caratterizzazione}, the set $\mathcal  F$ is a $U$-marked basis if and only if all the above normal forms vanish. Thus, we define the ideal $\mathscr U \subseteq \kk[C]$ generated by polynomials $Q_{i,\underline{\alpha},k}^{\underline\eta,\ell}(C)$ which occur as the coefficients of the terms in $\mathcal{N}(U)$ in the $F$-normal forms of the polynomials $x_i F_{\underline{\alpha},k}$.

\begin{theorem}\label{thm:rappr}\cite[Theorem 5.1]{ABRS} 
The marked functor $\MFFunctor{U}$ is represented by the scheme $\MFScheme{U}:=\mathrm{Spec}(\kk[C]/\mathscr{U})$.
\end{theorem}

The following result provides an immediate characterization of those $U$-marked bases whose $U$-resolution  is minimal.

\begin{corollary}\label{cor:closed scheme}
There exists a closed subscheme $\MFSchemeMin$ of  $\MFScheme{U}$ whose $\kk$-points parametrize the modules in $\MFFunctor{U}(\kk)$ that admit a minimal $U$-resolution.
\end{corollary}

\begin{proof} With the above notation, the vanishing of the polynomials $Q_{i,\underline{\alpha},k}^{\underline\eta,\ell}(C)$ is equivalent to the formation of the fundamental syzygies 
\[
x_i \ff_{\underline{\alpha},k} - \sum_{x^{\underline{\gamma}}\, \mee{\ell}{} \in \PP{U}} P^{\underline{\gamma},\ell}_{i,\underline{\alpha},k} \ff_{\underline\gamma,\ell}
\]
 of $\mathcal F$ in $R^{p_1}(- \dd')$, where $p_1=\vert\PP{U}\vert$. Let $\mathscr M$ be the ideal generated by the non-zero constants appearing as coefficients of terms in $R^{p_1}(- \dd')$ in these syzygies. Then, by Theorem \ref{prop:firstSyzForMin}, the closed subscheme $\MFSchemeMin=\mathrm{Spec}(\kk[C]/(\mathscr M+\mathscr U))$ is contained in $\MFScheme{U}$, and its $\kk$-points parametrize the modules in $\MFFunctor{U}(\kk)$ having a minimal $U$-resolution. 
Indeed, when $A=\kk$, Theorem~\ref{prop:firstSyzForMin} , implies that a marked basis belonging to $\MFFunctor{U}(\kk)$ yields a minimal $U$-resolution if and only if no non-zero constants appear among the coefficients of the terms in $R^{p_1}(- \dd')$ occurring in the fundamental syzygies.
\end{proof} 

\begin{remark}
~
\begin{itemize}
\item[(i)] Observe that if $A$ is not a field, it may happen that a $U$-marked basis $F$ admits a minimal $U$-resolution even though it does not correspond to a point of the closed subscheme $\MFSchemeMin$ introduced in Corollary \ref{cor:closed scheme}. This situation occurs, for instance, when a non-zero but non-invertible constant element of $A$ appears among the coefficients of the terms in $R^{p_1}(- \dd')$ in a fundamental syzygy of $F$.
\item[(ii)] Recalling the discussion in Remark \ref{rem: oss 1}, Corollary \ref{cor:closed scheme} implies that there exists a closed subscheme of a marked scheme over a quasi-stable ideal $U$ whose $\kk$-points parametrize the componentwise linear ideals in $\MFFunctor{U}(\kk)$. 
\end{itemize}
\end{remark}

We now generalize the notion of $U$-marked functor in order to parametrize, simultaneously, a $U$-marked basis together with its syzygies, or even its complete $U$-resolution. With the above notation, we denote by $U^{(t)} \subseteq R^{p_t}(-\dd^{(t)})$ the stable module such that the $t$-th syzygy module in a $U$-resolution is generated by a $U^{(t)}$-marked basis. By Theorem \ref{th:fundamental} and by the definition of $U$-resolution, such $U^{(t)}$-marked basis is precisely the set $F^{\partial_t}$ of the $t$-th fundamental syzygies of a $U$-marked basis $F$ of a module $M$. For convenience, we set $U^{(0)}=U$ and $F^{\partial_0}=F$.

\begin{definition}\label{def:FuncSyzRes}
Let $U=\bigoplus_{k=1}^m J^{(k)} \ee_k \subseteq R^m(- \dd)$ be a quasi-stable module. We define the {\em $U$-marked syzygy functor} as the covariant functor $\MSyzFunctor{U}: \underline{\text{Noeth~$\kk$-Alg}} \rightarrow \underline{\text{Sets}}$
that associates to any Noetherian $\kk$-algebra $A$ the set
\begin{equation}\label{eq:marked-syz-functor}
\MSyzFunctor{U}(A) :=  \left\{ (M,S_1)\ \middle\vert\ \begin{array}{l} M \oplus \langle\mathcal{N}(U)\rangle = R_A^m(-\dd) \\
S_1\text{~is the syzygy module of $F^{\partial_0}$} \end{array}\right\}.
\end{equation}

Analogously,\hfill we\hfill define\hfill the\hfill {\em $U$-marked\hfill resolution\hfill functor}\hfill as\hfill the\hfill covariant\hfill functor\\ $\MResFunctor{U}: \underline{\text{Noeth~$\kk$-Alg}} \rightarrow \underline{\text{Sets}}$
that associates to any Noetherian $\kk$-algebra $A$ the set
\begin{equation}\label{eq:marked-syz-functor2}
\MResFunctor{U}(A) :=  \left\{ (M,S_1,\ldots,S_t,\ldots, S_{n-D})\ \middle\vert\ \begin{array}{l} M \oplus \langle\mathcal{N}(U)\rangle = R_A^m(-\dd)\\  
S_{t}\text{~is the syzygy module of $F^{\partial_{t-1}}$,}\ \forall\ t > 0 \end{array}\right\}.
\end{equation}
where the integer $D$ is as in Theorem \ref{thm:freeRes}. 

In both cases, the functors associate to any morphism $\sigma:A \to A'$ the corresponding natural map induced by extension of scalars. Notice that, as in case of the marked functor $\MFFunctor{U}$, the image $\sigma(F^{\partial_0})$ is again a $U$-marked basis, and, by construction, each $\sigma(F^{\partial_{t}})$ is the set of fundamental syzygies of $\sigma(F^{\partial_{t-1}})$.
\end{definition}

\begin{proposition}\label{prop: isom first factor}
    Let $U \subseteq R^m(- \dd)$ be a quasi-stable module. Then, the projections onto the first factor $\pi_1:\MSyzFunctor{U} \to \MFFunctor{U}$ and $\pi_1:\MResFunctor{U} \to \MFFunctor{U}$ defined, for every Noetherian $\kk$-algebra $A$, by
    \[
    \begin{array}{ccc} 
    \MSyzFunctor{U}(A) & \xrightarrow{\pi_1(A)} & \MFFunctor{U}(A) \\
    (M,S_1) & \longmapsto & M
    \end{array}
    \qquad
    \begin{array}{ccc} 
    \MResFunctor{U}(A) & \xrightarrow{\pi_1(A)} & \MFFunctor{U}(A) \\
    (M,S_1,\ldots, S_{n-D}) & \longmapsto & M
    \end{array}
    \]
    are natural isomorphisms.
\end{proposition}

\begin{proof}
    The projection $\pi_1$ is a natural transformation from $\MSyzFunctor{U}$ (resp.~$\MResFunctor{U}$) to $\MFFunctor{U}$ since, for every morphism $\sigma: A \to A'$, the following diagrams commute:
    \begin{center}
        \begin{tikzpicture}[scale=0.9]
            \node (A) at (0,0) [inner sep=8pt] {$\MSyzFunctor{U}(A)$};
            \node (B) at (4,0) [inner sep=8pt] {$\MSyzFunctor{U}(A')$};
            \node (C) at (0,-2.5) [inner sep=8pt] {$\MFFunctor{U}(A)$};
            \node (D) at (4,-2.5) [inner sep=8pt] {$\MFFunctor{U}(A')$};
            \draw[->] (A) --node[above]{\tiny $\MSyzFunctor{U}(\sigma)$} (B);
            \draw[->] (C) --node[above]{\tiny $\MFFunctor{U}(\sigma)$} (D);
            \draw[->] (A) --node[left]{\tiny $\pi_1(A)$} (C);
            \draw[->] (B) --node[right]{\tiny $\pi_1(A')$} (D);

            \begin{scope}[shift={(9,0)}]
            \node (A) at (0,0) [inner sep=8pt] {$\MResFunctor{U}(A)$};
            \node (B) at (4,0) [inner sep=8pt] {$\MResFunctor{U}(A')$};
            \node (C) at (0,-2.5) [inner sep=8pt] {$\MFFunctor{U}(A)$};
            \node (D) at (4,-2.5) [inner sep=8pt] {$\MFFunctor{U}(A')$};
            \draw[->] (A) --node[above]{\tiny $\MResFunctor{U}(\sigma)$} (B);
            \draw[->] (C) --node[above]{\tiny $\MFFunctor{U}(\sigma)$} (D);
            \draw[->] (A) --node[left]{\tiny $\pi_1(A)$} (C);
            \draw[->] (B) --node[right]{\tiny $\pi_1(A')$} (D);
            \end{scope}
        \end{tikzpicture}
    \end{center}
    
To prove the claim, it is enough to show that, for every Noetherian $\kk$-algebra $A$, the map $\pi_1(A)$ is a bijection between $\MSyzFunctor{U}(A)$ (resp.~$\MResFunctor{U}(A)$) and $\MFFunctor{U}(A)$. 

Indeed, by Theorem \ref{th:fundamental},if $M$ is generated by a $U$-marked basis $F$, then there exists a unique $U^{(1)}$-marked basis $F^{\partial_1}$ generating its first syzygy module $S_1$. By Rrepeating this argument inductively, we obtain that each higher sygyzy module $S_t$ is uniquely determined by $M$. Therefore, $\pi_1(A)$ is bijective, and the naturality of $\pi_1$ implies that it is an isomorphism of functors.
\end{proof}

Thanks to the Yoneda Lemma, the following statement is straightforward.

\begin{corollary}\label{cor:isomorfismo}
    The functors $\MSyzFunctor{U}$ and $\MResFunctor{U}$ are representable, and the representing schemes $\MSyzScheme{U}$ and $\MResScheme{U}$ are isomorphic to $\MFScheme{U}$.
\end{corollary}

We now focus on determining explicit equations for the representing schemes $\MSyzScheme{U}$ and $\MResScheme{U}$, as well as on describing concretely the isomorphisms with $\MFScheme{U}$.
 
Taking into account \eqref{supporto2}, for every $F_{\underline{\alpha},k}\in \mathcal F$ and every non-multiplicative variable $x_i$ of $x^{\underline{\alpha}}$, we consider the set 
\[
{\sf s}_{i,\underline{\alpha},k}:= \bigcup_{x^{\underline{\gamma}}\,\mathbf{e}_\ell \in \mathcal{P}_U} \left\{x^{\underline{\eta}}\,\ff_{\underline{\gamma},\ell}\ \middle\vert\ x^{\underline{\eta}} \in \mathcal{C}_i(x^{\underline{\gamma}})_{d_{k} - d_{\ell} +1}\right\} \cap \mathcal N(U^{(1)})
\]
and we define the element
\begin{equation}\label{eq:presyz}
x_i \ff_{\underline{\alpha},k}-
\sum_{
\substack{x^{\underline{\eta}} \ff_{\underline{\gamma},\ell} \in {\sf s}_{i,\underline{\alpha},k}
 }}  
\pB{i}{\underline{\alpha}}{k}{\underline{\eta}}{\underline{\gamma}}{\ell} x^{\underline{\eta}} \ff_{\underline{\gamma},\ell}
\end{equation}
in $R^{p_1}_{\kk[B]}(-\dd(1))$, where $\kk[B]$ denotes the polynomial ring whose variables are the parameters in the set $B:= \bigcup B_{i,\underline{\alpha},k}$ with
\[
B_{i,\underline{\alpha},k}:=\left\{ \pB{i}{\underline{\alpha}}{k}{\underline{\eta}}{\underline{\gamma}}{\ell} \ \vert \ x^{\underline{\eta}} \ff_{\underline{\gamma},\ell} \in {\sf s}_{i,\underline{\alpha},k} \right\}.
\]
We call an element of $R^{p_1}_{\kk[B]}(-\dd^{(1)})$ of the form \eqref{eq:presyz} a \emph{fundamental pre-syzygy} of $\mathcal F$, and we denote by $\mathcal F_{\mathrm{pre-syz}}$ the set consisting of all such elements.

Next, in every fundamental pre-syzygy we replace $\ff_{\underline{\alpha},k}$ and $\ff_{\underline{\gamma},\ell}$ by the corresponding elements $F_{\underline{\alpha},k}$ and $F_{\underline{\gamma},\ell}$ of $\mathcal F$, obtaining the following homogeneous element in $R_{\kk[B,C]}^{m}(-\dd)$ 
\begin{equation}\label{eq:sostPreSyz}
x_i F_{\underline{\alpha},k}-
\sum_{
\substack{x^{\underline{\eta}} \ff_{\underline{\gamma},\ell} \in {\sf s}_{i,\underline{\alpha},k}
 }} 
\pB{i}{\underline{\alpha}}{k}{\underline{\eta}}{\underline{\gamma}}{\ell} x^{\underline{\eta}} F_{\underline{\gamma},\ell}.
\end{equation}
We denote by $\mathscr S$ the ideal in the polynomial ring $\kk[B,C]$ generated by the coefficients of the terms in $R^{m}(-\dd)$ occuring in the elements~\eqref{eq:sostPreSyz}. The following result is now almost immediate.

\begin{lemma}\label{lemma:contained}
With the notation above, and with $\mathscr U$ as in Theorem \ref{thm:rappr}, the ideal generated by $\mathscr U$ in the ring $K[B,C]$ is contained in $\mathscr S$.
\end{lemma}

\begin{proof}
It is enough to observe that in the quotient ring $\kk[B,C]/\mathscr S$ the polynomials generating $\mathscr U$ vanish, since the rewriting \eqref{eq:unique} is unique.
\end{proof}

Given a $U$-marked set $F=\{\mf{\underline{\alpha}}{k}\}\subset R^m_A(-\dd)$, we consider the evaluation morphism $\phi_{F}:\kk[B,C]\rightarrow A$ defined as follows. For every $a\in \kk$, we set $\phi_{F}(a)=a$. For every $\pC{\underline{\alpha}}{k}{\underline{\eta}}{\ell}\in C$, we let $\phi_{F}(\pC{\underline{\alpha}}{k}{\underline{\eta}}{\ell})$ be the coefficient of the term $x^{\underline{\eta}}\mee{\ell}{}$ in $\mf{\underline{\alpha}}{k}$. Now, for every $x_{\underline{\alpha}}{\bf e}_k\in \PP{U}$ and for every non-multiplicative variable $x_i$ of $x^{\underline {\alpha}}$, we consider the element
\[
x_i f_{\underline{\alpha},k} - \sum_{x^{\underline{\gamma}}\, \mee{\ell}{} \in \PP{U}} P^{\underline{\gamma},\ell}_{i,\underline{\alpha},k}(\phi_F(C)) f_{\underline\gamma,\ell}
\]
which is obtained from \eqref{eq:unique} by replacing each $F_{\underline{\gamma},\ell}$ with the corresponding $f_{\underline{\gamma},\ell}$. 
Then the morphism $\phi_{F}$ associates to each parameter $\pB{i}{\underline{\alpha}}{k}{\underline{\eta}}{\underline{\gamma}}{\ell}$ the coefficient 
of $x^{\underline{\eta}} f_{\underline{\gamma},\ell}$ in the above expression. We extend $\phi_{F}$ to $\kk[B,C]$ in the natural way. 

Finally, we extend $\phi_{F}$ to $R_{\kk[C]}^m(-\dd)$ (resp.~to $R_{\kk[B,C]}^{p_1}(-\dd^{(1)})$) by letting $\phi_{F}(f)$ act only on the coefficients of $f$, which belong to $\kk[C]$ (resp.~to $\kk[B,C]$).

\begin{theorem}\label{thm:reprFuncSyz}
The functor $\MSyzFunctor{U}$ is represented by the scheme $\MSyzScheme{U}:=\Spec(\kk[B,C]/\mathscr{S})$.
\end{theorem}

\begin{proof} 
It is enough to prove that every $U$-marked set $F$ is a $U$-marked basis and that $F^{\partial_1}=\phi_{F}(\mathcal F_{\mathrm{pre-syz}})$ if and only if $\phi_{F}$ factors through $ \kk[B,C]/\mathscr S$, or in other words if and only if the following diagram commutes:
\[
\xymatrix{
\kk[B,C]  \ar[rr]^{\phi_{F}} \ar[dr]& &A\\
& \kk[B,C]/\mathscr S\ar[ru]}\,.
\]
Equivalently, we must show that $F$ is a $U$-marked basis if and only if $\ker(\phi_{F})$ contains $\mathscr S$.

If $F$ is a $U$-marked basis, then the polynomials corresponding to those of \eqref{eq:sostPreSyz} via $\phi_{F}$ vanish identically, hence $\mathscr S\subseteq \ker(\phi_{F})$. 

Now, suppose that $F$ is a $U$-marked set and that $\mathscr S\subseteq \ker(\phi_{F})$. 
Applying $\phi_{F}$ to the polynomials in \eqref{eq:sostPreSyz}, we obtain zero polynomials; therefore, these yeild the fundamental syzygies of the marked set $F$. By Lemma \ref{lemma:contained} and Theorem \ref{th:caratterizzazione}, it follows that $F$ is indeed a $U$-marked basis.
\end{proof}

\begin{corollary}\label{cor:eliminazione}
With the above notation, there exists a projection morphism $\mathbb{A}^{\vert C \cup B\vert}\to \mathbb{A}^{\vert C\vert}$ whose restriction to $\MSyzScheme{U}$ induces an isomorphism $\MSyzScheme{U} \simeq \MFScheme{U}$ and in particular, $\mathscr S\cap \kk[C]=\mathscr U$.
\end{corollary} 

\begin{proof} 
Since $\mathcal F$ is a $U$-marked set, the rewriting procedure $\xrightarrow{\mathcal F^{\ast}}_\star$ provides unique polynomials $ b_{i,\underline{\alpha},k}^{\underline{\eta},\underline{\gamma},\ell} \in \kk[C]$ such that 
\[
x_i F_{\underline{\alpha},k} = \sum_{
\substack{x^{\underline{\eta}} \ff_{\underline{\gamma},\ell} \in {\sf s}_{i,\underline{\alpha},k}
 }} 
b_{i,\underline{\alpha},k}^{\underline{\eta},\underline{\gamma},\ell} x^{\underline{\eta}} F_{\underline{\gamma},\ell} +G_{\underline{\alpha},k}^i,
\]
where $G_{\underline{\alpha},k}^i=0$ in $\kk[C]/\mathscr U$. 

Let $\mathscr B$ denote the ideal generated in $\kk[B,C]$ by the polynomials $B_{i,\underline{\alpha},k}^{\underline{\eta},\underline{\gamma},\ell}-b_{i,\underline{\alpha},k}^{\underline{\eta},\underline{\gamma},\ell}$. Then, in the quotient ring $\kk[B,C]/\mathscr S$, these generators vanish because the rewriting above is unique. Hence, $\mathscr U \cdot K[B,C] +\mathscr B \subseteq \mathscr S$ thanks also to Lemma~\ref{lemma:contained}.
 
Arguing as in the proof of Theorem \ref{thm:reprFuncSyz}, we deduce that $\Spec(K[B,C]/(\mathscr U \cdot K[B,C] +\mathscr B )$ represents $\MSyzFunctor{U}{}$, and by Yoneda Lemma this scheme is isomorphic to $\Spec(\kk[B,C]/\mathscr S)$. 
Therefore, we conclude that $\mathscr U \cdot K[B,C] +\mathscr B= \mathscr S$, and hence $\mathscr S\cap \kk[C]=\mathscr U$.

This isomorphism is induced by the projection $\kk[B,C]\rightarrow\kk[C]/\mathscr U$ that sends each $b_{i,\underline{\alpha},k}^{\underline{\eta},\underline{\gamma},\ell}$ to its corresponding $B_{i,\underline{\alpha},k}^{\underline{\eta},\underline{\gamma},\ell}$ and fixes the parameters $C$. The kernel of this projection coincides with the ideal $\mathscr U \cdot K[B,C] +\mathscr B$, yielding the claimed isomorphism $\MSyzScheme{U}\simeq\MFScheme{U}$.
\end{proof}

It is now useful to give a more explicit interpretation of Theorem \ref{thm:reprFuncSyz} in terms of matrix products. Let us consider the following matrices:  
\begin{itemize}
    \item $M_{0}$ is the $m \times p_1 $ matrix whose columns are the elements of $\mathcal F\subseteq R_A^m(-\dd)$; 
    \item $M_{1}$ is the $p_1 \times \vert \PP{U^{(1)}}\vert$ matrix whose columns are the elements of $\mathcal F_{\mathrm{pre-syz}}\subseteq R_A^{p_1}(-\dd^{(1)})$.
\end{itemize}
\begin{corollary}\label{cor:prodotto matrici}
With the notation above, $\mathcal F$ is a $U$-marked basis and $\mathcal F_{\mathrm{pre-syz}}$ is the corresponding set of fundamental syzygies if and only if $M_{0} M_{1}=0$. 
\end{corollary}

\begin{proof}
By construction, the entries of the matrix product $M_{0}  M_{1}$ are precisely the coefficients of the terms appearing in the polynomials of~\eqref{eq:sostPreSyz}. Hence, the condition $M_{0} M_{1} = 0$ is equivalent to the vanishing of all such polynomials, which is exactly the requirement that $\mathcal F_{\mathrm{pre-syz}}$ consists of the fundamental syzygies of the $U$-marked basis $\mathcal F$.
\end{proof}

\begin{example}
Consider the stable ideal $J=(x_2,x_1^2)\subseteq \kk[x_0,x_1,x_2]$, which defines a complete intersection corresponding to a point of the Hilbert scheme $\mathrm{Hilb}^2(\mathbb P^2)$. The $J$-resolution is minimal
\[
0 \to R_A(-3) \overset{\partial^J_1}{\longrightarrow} R_A(-1)\oplus R_A(-2) {\buildrel \partial^J_0\over\longrightarrow} J \to 0
\]
where the matrix of $\partial^J_0$ is $M^J_0=\left[\begin{array}{cc} x_2 &  x_1^2\end{array}\right]$ and the matrix of $\partial^J_1$ is $M^J_1=\left[\begin{array}{c} -x_1^2 \\  x_2\end{array}\right]$. Given the $J$-marked set $\mathcal F=\{F_1= \HTemph{x_2}-c_1x_1-c_2x_0, F_2= \HTemph{x_1^2}-c_3x_1x_0-c_4x_0^2\}$, we construct $\mathcal F_{\mathrm{pre-syz}}$ and then the matrices
\[
M_0=\left[\begin{array}{cc} F_1 &  F_2\end{array}\right], \quad M_1=\left[\begin{array}{c}  -b_1x_1^2-b_2x_1x_0-b_3x_0^2 \\   \HTemph{x_2}-b_4x_1-b_5x_0\end{array}\right]
\]
according to Corollary \ref{cor:suppSyz}. Here $\PP{U}=\{x_2,x_1^2\}$, $\PP{U^{(1)}}=\{x_2 {\bf f}_2\}$ and $\mathcal N(U^{(1)})=\T{\bf f}_1\cup\{x_0^{\alpha_0}x_1^{\alpha_1} : (\alpha_0,\alpha_1)\in \mathbb Z_{\geqslant 0}^2 \}{\bf f}_2$. 

Since $\mathcal F$ is a $J$-marked basis, we have $\MFScheme{U}=\mathbb A^4$. Moreover, note that $\MFScheme{U^{(1)}}=\mathbb A^8$, but $\MSyzScheme{U}$ is isomorphic to $\MFScheme{U}$ by Corollary \ref{cor:isomorfismo}; hence $U^{(1)}$-marked bases are not necessarily composed of the fundamental syzygies of a $U$-marked basis. Further, by Corollary~\ref{cor:prodotto matrici}, we obtain $b_1=-1, b_2=-c_3, b_3=-c_4, b_4=c_1, b_5=c_2$. 
\end{example}

Following the previous notation, for each $t>0$ let $M_{t}$ be the matrix whose columns are the pre-syzygies of $\mathcal F^{\partial_{t-1}}$, with the set of parameters $B^{(t)}$ defined analogously to $B$. Observe that, by Theorem \ref{th:fundamental}, the condition $M_{t-1}M_t=0$ implies the condition $M_tM_{t+1}=0$.

For every $t> 0$, let $\mathscr S^{(t)}$ be the ideal generated by the coefficients of the terms in $\mathbb T$ appearing in the entries of the product of matrices $M_{t-1}M_{t}$. In particular, we have $\mathscr S^{(1)}=\mathscr S$ and $B^{(1)}=B$.

\begin{proposition}\label{prop:elimSyzn}
The functor $\MResFunctor{U}$ is represented by the scheme 
\[
\MResScheme{U}:=\mathrm{Spec}\left(\kk\left[C,B^{(1)},\dots,B^{(n-D)}\right]/\mathscr S^{(1)}+\dots+\mathscr S^{(n-D)}\right),
\]
where $D$ is as in Theorem \ref{thm:freeRes}.
An explicit isomorphism between $\MResScheme{U}$ and $\MFScheme{U}$ is obtained by recursively eliminating the parameters in each $B^{(t)}$, i.e., $\mathscr U=(\mathscr S^{(1)}+\dots +\mathscr S^{(n-D)})\cap \kk[C]$. 
\end{proposition}

\begin{proof}
For representability, it suffices to extend the arguments of the proof of Theorem \ref{thm:reprFuncSyz} to all steps of the $U$-resolution. For the isomorphism, for each $t>0$, the parameters in $B^{(t)}$ can be eliminated analogously to the elimination of the parameters in $B$ as carried out in the proof of Corollary \ref{cor:isomorfismo}. 
\end{proof}

\begin{example}\label{ex:funtori vari}
Consider the stable ideal $J=(x_3,x_2^2,x_2x_1)\subseteq  K[x_0,x_1,x_2,x_3]$. The $J$-resolution is
\[
0 \to R_A(-4) {\buildrel \partial^J_2\over\longrightarrow} R_A^3(-3) \overset{\partial^J_1}{\longrightarrow} R_A(-1)\oplus R_A(-2)^2 {\buildrel \partial^J_0\over\longrightarrow} J \to 0
\]
with matrices 
\[
M^J_0=\left[\begin{array}{ccc} \HTemph{x_3} & \HTemph{ x_2^2} & \HTemph{x_2x_1}\end{array}\right],\quad M^J_1=\left[\begin{array}{ccc} 0 & -x_2x_1 & -x_2^2\\ -x_1 & 0 & \HTemph{x_3} \\\HTemph{ x_2}  & \HTemph{x_3} & 0\end{array}\right],\quad M^J_2=\left[\begin{array}{c} \HTemph{x_3} \\ -x_2 \\ x_1\end{array}\right].
\]
Given the $J$-marked set $\mathcal F=\{F_1= \HTemph{x_3}-c_1x_2-c_2x_1-c_3x_0,\ F_2= \HTemph{x_2^2}- c_4x_1^2-c_5x_2x_0-c_6x_1x_0-c_7x_0^2,\ F_3= \HTemph{x_1x_2}-c_8x_1^2-
c_9x_2x_0-c_{10}x_1x_0-c_{11}x_0^2\}$, we construct $\mathcal F_{\mathrm{pre-syz}}$ and the matrices 
\[
\begin{split}
&M_0=\left[\begin{array}{ccc} f_1 &  f_2 & f_3\end{array}\right],\qquad M_2=\left[\begin{array}{c}  \HTemph{x_3}-d_{1}x_2-d_{2}x_1-d_{3}x_0 \\ -d_{4}x_2-d_{5}x_1-d_{6}x_0 \\ 
-d_{7}x_2-d_{8}x_1-d_{9}x_0\end{array}\right]\\
&M_1=\left[\begin{array}{ccc} -b_1x_1^2-b_2x_1x_0-b_{3}x_0^2 & -g_2 & -g_3\\ 
-b_{16}x_1-b_{17}x_0 & -b_{18}x_2-b_{19}x_1-b_{20}x_0 &  \HTemph{x_3}-b_{21}x_2-b_{22}x_1-b_{23}x_0 \\ 
 \HTemph{x_2}-b_{24}x_1-b_{25}x_0  &  \HTemph{x_3}-b_{26}x_1-b_{27}x_0 & -b_{28}x_1-b_{29}x_0 \end{array}\right]\\
\end{split}
\]
with $g_2=b_{4}x_2^2 + b_{5} x_2x_1+ b_{6}x_1^2+b_{7}x_2x_0+b_{8}x_1x_0+b_{9}x_0^2$ and $g_3=b_{10}x_2^2 + b_{11} x_2x_1+ b_{12}x_1^2+b_{13}x_2x_0+b_{14}x_1x_0+b_{15}x_0^2$, 
according to Corollary \ref{cor:suppSyz}. Here, $\PP{U}=\{x_3,x_2^2,x_2x_1\}$, $\PP{U^{(1)}}=\{x_3 {\bf f}_2, x_3{\bf f}_3, x_2{\bf f}_3\}$ and $\mathcal N(U^{(1)})=\T{\bf f}_1 \cup \{x_0^{\alpha_0}x_1^{\alpha_1}x_2^{\alpha_2} : (\alpha_0,\alpha_1,\alpha_2)\in \mathbb Z_{\geqslant 0}^3 \}{\bf f}_2 \cup \{x_0^{\alpha_0}x_1^{\alpha_1} : (\alpha_0,\alpha_1)\in \mathbb Z_{\geqslant 0}^2 \}{\bf f}_3$, and $\PP{U^{(2)}}=\{x_3 {\bf f}_{x_3,1}\}$.

Imposing the conditions $M_0M_1=0$ yields the generators of the ideal $\mathscr S^{(1)}$ defining the scheme $\MSyzScheme{U}$, and imposing $M_1M_2=0$ yields the generators of the ideal $\mathscr S^{(2)}$, so that $\mathscr S^{(1)}+ \mathscr S^{(2)}$ defines $\MResScheme{U}$.
\end{example}

\subsection{The projection onto the second factor }
\label{subsec:projection onto the second factor}

Up to now, we have observed that the parameters in $B,\ldots,B^{(t)},\ldots$, can be expressed in terms of the parameters in the set $C$ via an elimination procedure, which provides explicit isomorphisms between $\MSyzScheme{U}$ and $\MFScheme{U}$ and between $\MResScheme{U}$ and $\MFScheme{U}$. 

In the following, we exploit a key result from \cite{BE1974} to show that, under the assumptions that the quasi-stable module $U$ is a quasi-stable ideal $J$ and that $\textnormal{depth}(J)\geqslant 2$, the parameters in $C$ can also be expressed in terms of those in $B$. 

\begin{lemma}\label{lemma:depth}
Let $\mathcal I\subseteq (\kk[C]/\mathscr U)[X]$ be the ideal generated by the $J$-marked basis $\mathcal F$ over $\mathcal R= (\kk[C]/\mathscr U)[X]$. Then, $\textnormal{depth}(J)\leqslant \textnormal{depth}(\mathcal I)$.
\end{lemma}

\begin{proof}Recall that the depth of an $\mathcal R$-module $M$ is the maximal length of a regular sequence in $\mathcal R$ consisting of elements of $M$ (see, e.g., \cite[Section 18.1]{E}). Moreover, by the Auslander- Buchsbaum formula, if $M$ admits a finite free resolution, we have $\textnormal{depth}(M) + \textnormal{pdim}(M) =  \textnormal{depth}(\mathcal R)$.   Both $J$ and $\mathcal I$ admit a $J$-resolution, so it follows from Corollary \ref{cor:maggiorazioni} that $\textnormal{depth}(J)\leqslant \textnormal{depth}(\mathcal I)$.
\end{proof}

\begin{definition}\label{def:second factor}
Let $U=\bigoplus_{k=1}^m J^{(k)} \ee_k \subseteq R^m(- \dd)$ be a quasi-stable module. We define the {\em $U$-marked second projection functor} as the covariant functor $\MSyzFunctorDue{U}: \underline{\text{Noeth~$\kk$-Alg}} \rightarrow \underline{\text{Sets}}$
that associates to any Noetherian $\kk$-algebra $A$ the set
\begin{equation}\label{eq:marked-syzDue-functor}
\MSyzFunctorDue{U}(A) :=  \left\{ S_1\ \middle\vert\ \begin{array}{l} M \oplus \langle\mathcal{N}(U)\rangle = R_A^m(-\dd) \\
S_1\text{~is the syzygy module of $F^{\partial_0}$} \end{array}\right\}.
\end{equation}
This functor associates to any morphism $\sigma:A \to A'$ the natural map induced by extension of scalars. As with the other functors, $\sigma(F^{\partial_{1}})$ is the set of fundamental syzygies of $\sigma(F^{\partial_{0}})$.
\end{definition}

\begin{theorem}\label{th:representation second factor}
With the above notation, the functor $\MSyzFunctorDue{U}$ is represented by the scheme $\MSyzSchemeDue{U}:=\Spec(\kk[B]/\mathscr{S}\cap \kk[B])$.
\end{theorem}

\begin{proof}
Arguing as in the proof of Theorem \ref{thm:reprFuncSyz} and considering the restriction $\phi'_F$ of the evaluation homomorphism $\phi_F$ to $\kk[B]$, it suffices to show that a $U^{(1)}$-marked set is the $U^{(1)}$-marked basis $F^{\partial_1}=\phi_{F}(\mathcal F_{\mathrm{pre-syz}})$ of the first syzygy module of a $U$-marked basis $F$ if and only if $\phi_{F}$ factors through $\kk[B]/(\mathscr S\cap \kk[B])$, i.e., if and only if the following diagram commutes:
\[
\xymatrix{
\kk[B]  \ar[rr]^{\phi_{F}} \ar[dr]& &A\\
& \kk[B]/(\mathscr S\cap \kk[B])\ar[ru]}\,.
\]
Equivalently, we need to prove that $\phi'_{F}(\mathcal F_{\mathrm{pre-syz}})$ coincides with $F^{\partial_1}$ for a $U$-marked basis $F$ if and only if $\ker(\phi'_{F})$ contains $\mathscr S\cap \kk[B]$, since $\ker(\phi'_F)=\ker(\phi_F)\cap \kk[B]$.

If $F^{\partial_1}=\phi_{F}(\mathcal F_{\mathrm{pre-syz}})$ for a $U$-marked basis $F$, then $\mathscr S \subseteq \ker(\phi_F)$, by Theorem \ref{thm:reprFuncSyz}, and hence $\mathscr S \cap \kk[B] \subseteq \ker(\phi_F)\cap \kk[B]=\ker(\phi'_F)$.

Suppose now that $F$ is a $U$-marked set and that $\mathscr S\cap \kk[B]\subseteq \ker(\phi'_{F})$. 
Applying $\phi'_{F}$ to the elements in \eqref{eq:presyz} produces the fundamental syzygies of the marked set $F$, which in turn ensures that $F$ is a $U$-marked basis by Lemma \ref{lemma:contained} and Theorem~\ref{th:caratterizzazione}. 
\end{proof}

Like for the projections onto the first factor, the projections onto the second factor $\pi_2:\MSyzFunctor{U} \to \MSyzFunctorDue{U}$ and $\pi_2:\MResFunctor{U} \to \MSyzFunctorDue{U}$ defined for every $\kk$-algebra $A$ by
    \[
    \begin{array}{ccc} 
    \MSyzFunctor{U}(A) & \xrightarrow{\pi_1(A)} & \MSyzFunctorDue{U}(A) \\
    (M,S_1) & \longmapsto & S_1
    \end{array}
    \qquad
    \begin{array}{ccc} 
    \MResFunctor{U}(A) & \xrightarrow{\pi_1(A)} & \MSyzFunctorDue{U}(A) \\
    (M,\ldots,S_t,\ldots) & \longmapsto & S_1
    \end{array}
    \]
are natural transformations from $\MSyzFunctor{U}$ (resp.~$\MResFunctor{U}$) to $\MSyzFunctorDue{U}$.

However, unlike the projections on the first factor (see Proposition \ref{prop: isom first factor}), these natural transformations are not necessarily isomorphisms, as the following examples illustrate.

\begin{example}\label{ex:no inj}
Consider the quasi-stable ideal $J = (x_2,x_1^2)$ in $\kk[x_0,x_1,x_2]$ with $x_2>x_1>x_0$, and the $J$-marked basis $F= \{ \HTemph{x_2}-x_0, \HTemph{x_1^2} - x_0^2\}$. The depth of $J$ is $2$. The module of the first syzygies of $F$ is generated by the $U^{(1)}$-marked basis consisting of the single syzygy $(-x_1^2 + x_0^2)\ff_{x_2}+ \HTemph{x_2\ff_{x_1^2}} - x_0\ff_{x_1^2}$.

Now consider the quasi-stable ideal $J'= x_2J = (x_2^2,x_2x_1^2)$, which has depth $1$, and the two $J'$-marked bases $F'=\{  \HTemph{x_2^2}-x_2x_0,  \HTemph{x_2x_1^2} - x_2x_0^2\}$ and $F''= \{(x_2+x_0)(x_2-x_0), (x_2+x_0)(x_1^2 - x_0^2)\} = \{ \HTemph{x_2^2} - x_0^2,  \HTemph{x_2x_1^2} - x_2x_0^2 + x_1^2x_0 - x_0^3\}$. It is immediate to verify that $F'^{\partial_1}=F''^{\partial_1}=F^{\partial_1}$. Hence, for $J'$-marked bases, the projection onto the second factor is not injective.
\end{example} 

\begin{example}\label{ex:no inj 2}Let $\kk$ be a field of characteristic 0.
For the quasi-stable ideal $J:=(x_3^2,x_3x_2)\subseteq \kk[x_0,\dots,x_3]$ with depth $1$, we compute the ideals $\mathscr U$ and $\mathscr S$ in $\kk[C]$ and $\kk[B,C]$ with $\vert C\vert=16$ and $\vert B\vert=6$, respectively. The Krull-dimension of both $\kk[C]/\mathscr U$ and $\kk[B,C]/\mathscr S$ is $7$. However, the Krull-dimension of $\kk[B]/(\mathscr S\cap \kk[B])$ is $6$, showing that $\kk[B,C]/\mathscr S$ and $\kk[B]/(\mathscr S\cap \kk[B])$ are not isomorphic. 
\end{example}

From now on, we consider a quasi-stable ideal $J\subseteq R_A$ of depth $\geqslant 2$.

Following the notation in \cite{BE1974}, if $P$ and $Q$ are free modules of finite rank over a commutative ring $R$, then by choosing bases for $P$ and $Q$ we can identify a map $f: P\to Q$ with a matrix, and denote by $I(f)$ the ideal generated by the non-zero minors of maximal size, i.e., the rank of $f$. By $P^\star$ we denote $\textnormal{Hom}_R(P,R)$.

\begin{theorem}\label{th:isomorphism 2}
Let $J\subseteq R_A$ be a quasi-stable ideal of depth $\geqslant 2$, and let $\mathcal I$ the ideal defining the universal family of the marked functor $\MFFunctor{J}$. Then, for every parameter $c$ in the set $C$, there exists a polynomial $q_c\in \kk[B]$ such that $c=q_c$ in the ring $\kk[B,C]/\mathscr S$, and the projection onto the second factor $\pi_2:\MSyzFunctor{J} \to \MSyzFunctorDue{J}$ is an isomorphism. In particular, $\kk[B,C]/\mathscr S$ and $\kk[B]/(\mathscr S\cap \kk[B])$ are isomorphic.
\end{theorem}

\begin{proof}
The\hfill ideal\hfill $\mathcal I\subseteq (\kk[C]/\mathscr U)[X]$\hfill is\hfill  generated\hfill by\hfill the\hfill $J$-marked\hfill basis\hfill $\mathcal F$\hfill over\hfill the\hfill ring\\ $(\kk[C]/\mathscr U)[X]\simeq (\kk[B,C]/\mathscr S)[X]$. By Lemma \ref{lemma:depth} and by the hypothesis, we have $\textnormal{depth}(\mathcal I)\geqslant \textnormal{depth}(J)\geqslant 2$. 

Since $\mathcal F$ is a $J$-marked basis, the ideal $\mathcal I$ admits a $J$-resolution, which is a particular finite free resolution over $(\kk[B,C]/\mathscr S)[X]$. For convenience, we write it as
\[
0 \to P_n \xrightarrow{f_n} P_{n-1} \xrightarrow{f_{n-1}} \dots \xrightarrow{f_2} P_1 \xrightarrow{f_1} P_0=(\kk[B,C]/\mathscr S)[X].
\]
The homomorphism $f_1$ is represented by the matrix $M_0$, written in terms of the parameters $C$ modulo the ideal $\mathscr U$, and $f_2$ is the homomorphism described by the matrix $M_1$, which can be expressed in terms of both $C$ and $B$, but even solely in terms of the parameters $B$ modulo the ideal $\mathscr S\cap \kk[B]$, thanks to Theorem \ref{th:representation second factor}. Similarly, for every $t>1$, the matrix $M_t$ can be described in terms of the parameters $B$, modulo $\mathscr S\cap \kk[B]$.

By \cite[Theorem 3.1]{BE1974}, for each integer $k$ such that $1\leqslant k < n$, there exists a unique homomorphism $a_k: (\kk[B,C]/\mathscr S)[X] \to \wedge^{r_k} P_{k-1}=\wedge^{r_k-1} P^\star_{k-1}$, where $r_k$ is the rank of $f_k$, such that the following diagrams commutes:
\[
\xymatrix{
\wedge^{r_k} P_k \ar[rr]^{\wedge^{r_k} f_k} \ar[dr]_{a^\star_{k+1}}& &\wedge^{r_k}P_{k-1}\\
& (\kk[B,C]/\mathscr S)[X] \ar[ru]_{a_k}}\,.
\]
The maps $a_k$ are constructed recursively starting from $a_n:=\wedge^{r_n}f_n$. There is a unique map $a$ making the following diagram commute (see \cite[page 94]{BE1974}):
\[
\xymatrix{
\wedge^{r_{k-1}} P^\star_{k-2} \ar[rr]^{\wedge^{r_{k-1}} f^\star_{k-1}} \ar[dr]_{a}& &\wedge^{r_{k-1}}P^\star_{k-1}\\
& (\kk[B,C]/\mathscr S)[X] \ar[ru]_{a_k}}\,.
\]
Then $a_{k-1}:=a^\star$ depends on $a_k$ and $\wedge^{r_{k-1}} f^\star_{k-1}$. In particular, $a_2$ can be described in terms of $a_3$ and $\wedge^{r_2} f^\star_2$,  hence in terms of the parameters $B$.

By \cite[Corollary 5.2]{BE1974}, we have $\mathcal I=h I(a_2)$, where $h$ is the generator of $I(a_1)$ and must be a constant, otherwise $\mathcal I$ would have depth $1$, contradicting the hypothesis. Hence, the ideal $\mathcal I$ depends only on the parameters $B$, modulo $\mathscr S\cap \kk[B]$, because $I(a_2)$ does. 

In conclusion, the generators of $\mathcal I$ contained in $\mathcal F$, modulo $\mathscr U$, can be expressed in terms of the parameters $B$, modulo $\mathscr S\cap \kk[B]$. This defines a map
\[
\Phi: \kk[C]/\mathscr U \to \kk[B]/(\mathscr S\cap \kk[B])
\]
that associates to each parameter $c$ of the set $C$ a corresponding polynomial $q_c \in \kk[B]$. Moreover, $\Phi$ is an isomorphism because both compositions $\Phi \circ \Phi'_F$ and $\Phi'_F\circ \Phi$ are the identity maps, where $\Phi'_F: \kk[B]/(\mathscr S\cap \kk[B]) \to \mathbb K[C]/\mathscr U=:A$ is the map introduced in the proof of Theorem~\ref{th:representation second factor}.
\end{proof}

\begin{example}
In this example, we illustrate that Theorem \ref{th:isomorphism 2} does not hold if the quasi-stable module $U$ is not an ideal. We consider hyperplane sections by $x_0$ of the modules of Example~\ref{ex:funtori vari}. In particular we take $U\subseteq R(-1)\oplus R(-2)^2$, with $\PP{U}=\{x_2 {\bf f}_3, x_3{\bf f}_3, x_3{\bf f}_2\}$ and $\mathcal N(U)=\T{\bf f}_1 \cup \{x_1^{\alpha_1}x_2^{\alpha_2} : (\alpha_1,\alpha_2)\in \mathbb Z_{\geqslant 0}^2 \}{\bf f}_2 \cup \{x_1^{\alpha_1} : \alpha_1\in \mathbb Z_{\geqslant 0} \}{\bf f}_3$, and moreover $\PP{U^{(1)}}=\{x_3 {\bf f}_{x_3,1}\}$.

The elements of the $U$-marked set $\mathcal F$  in $R_{\kk[C]}(-1)\oplus R_{\kk[C]}(-2)^2$, with $\kk[C]=\kk[c_1,\dots,c_{28}]$, 
 are:
\[
\begin{split}
    F_1={}&\HTemph{x_2\ff_{3}} + (c_{19}x_1^2-c_{20}x_1x_2-c_{21}x_1x_3-c_{22}x_2^2-c_{23}x_2x_3-c_{24}x_3^2)\ff_{1}+{}\\
    &( -c_{25}x_1-c_{26}x_2)\ff_{2} -c_{28}x_1\ff_{3},\\
    F_2={}&\HTemph{x_3\ff_3} + (-c_{10}x_1^2-c_{11}x_1x_2-c_{12}x_1x_3-c_{13}x_2^2-c_{14}x_2x_3-c_{15}x_3^2)\ff_1+{}\\ & +( -c_{16}x_1-c_{17}x_2)\ff_2 -c_{18}x_1\ff_3,\\
    F_3={}&\HTemph{ x_3\ff_2} + (-c_1x_1^2-c_2x_1x_2-c_3x_1x_3-c_4x_2^2-c_5x_2x_3-c_6x_3^2)\ff_{1}+{}\\
    &(-c_7x_1-c_8x_2)\ff_2 -c_9x_1\ff_3.
\end{split}
\]
These form the columns of the matrix $M_0$, while the matrix $M_1$ is
\[
M_1=\left[\begin{array}{c} \HTemph{x_3}-b_1x_1-b_2x_2 \\ -b_3x_1-b_4x_2 \\ -b_5x_1-b_6x_2 \end{array}\right].
\]
By imposing $M_0 M_1=0$ and applying Corollary \ref{cor:prodotto matrici}, we obtain the ideal $\mathscr S\subseteq \kk[B,C]$. In this case $\kk[B,C]/\mathscr S$ and $\kk[B]/(\mathscr S \cap \kk[B])$ are not isomorphic, since direct computations give $\dim (\kk[B,C]/\mathscr S)=16$ and $\dim (\kk[B]/(\mathscr S \cap \kk[B]))=4$.  
\end{example}

\appendix

\section{An example on locally Cohen-Macaulay curves}\label{sec:finalEx}

We conclude the paper with a detailed example concerning  a Hilbert scheme of $1$-dimension-al subschemes in the projective space $\mathbb{P}^3$. From now on, we assume that the base the field $\mathbb{K}$ is algebraically closed of characteristic 0. 
The explicit computations were performed using \textit{Macaulay2} \cite{M2}, and the commands to reproduce our results are provided in the ancillary file \href{http://www.paololella.it/software/example-curves-P3.m2}{\tt example-base.m2}. Additional ancillary files introduced below contain the results.

\smallskip

We\hfill consider\hfill the\hfill Hilbert\hfill scheme\hfill $\mathbf{Hilb}^{3t+2}(\mathbb{P}^3)$,\hfill which\hfill parametrizes\hfill 1-dimensional\hfill sub-\\ schemes of $\mathbb{P}^3$ with degree $3$ and arithmetic genus $-1$. This scheme has three irreducible components.
\begin{enumerate}
    \item The first component $\mathcal{C}_1$ is rational of dimension 18, and its generic point corresponds to the disjoint union of a plane cubic and two points (see Figure \ref{fig:example}, last row, left column).
    \item The second component $\mathcal{C}_2$ is rational of dimension 15, and its generic point corresponds to the disjoint union of a twisted cubic and one point (see Figure \ref{fig:example}, last row, center column).
    \item The third component $\mathcal{C}_3$ is rational of dimension 12, and its generic point corresponds to the disjoint union of a plane conic and a line (see Figure \ref{fig:example}, last row, right column).
\end{enumerate}

\begin{figure}[!hb]
        \begin{center}
        \begin{tikzpicture}

\tdplotsetmaincoords{65}{140}
\begin{scope}[shift={(1.2,-16.75)},scale=0.5,tdplot_main_coords,rotate=90]



  \filldraw[fill=black!20, draw=none, opacity=0.25]
    (-2.1,0,-2.1) -- (-2.1,0,2.1) -- (2.1,0,2.1) -- (2.1,0,-2.1) -- cycle;

    \filldraw[fill=black!20, draw=none, opacity=0.25]
    (-2.1,-2.1,0) -- (-2.1,2.1,0) -- (2.1,2.1,0) -- (2.1,-2.1,0) -- cycle;

  \filldraw[fill=black!20, draw=none, opacity=0.25]
    (0,-2.1,-2.1) -- (0,2.1,-2.1) -- (0,2.1,2.1) -- (0,-2.1,2.1) -- cycle;

\draw[] (-2.5,0,0) -- (2.5,0,0); 
  \draw[] (0,-2.5,0) -- (0,2.5,0); 
  \draw[] (0,0,-2.5) -- (0,0,2.5); 

 \filldraw[red] (1,-1,1) circle (1pt);
 \filldraw[red] (1,-1,-1) circle (1pt);
 \draw[thick, red]
    (0,-2,2.08008) -- (0,-1.99,2.07084) -- (0,-1.98,2.06161) -- (0,-1.97,2.0524) -- (0,-1.96,2.04319) -- (0,-1.95,2.03399) -- (0,-1.94,2.02481) -- (0,-1.93,2.01563) -- (0,-1.92,2.00647) -- (0,-1.91,1.99732)
      -- (0,-1.9,1.98818) -- (0,-1.89,1.97905) -- (0,-1.88,1.96994) -- (0,-1.87,1.96084) -- (0,-1.86,1.95175) -- (0,-1.85,1.94267) -- (0,-1.84,1.93361) -- (0,-1.83,1.92456) -- (0,-1.82,1.91553) --
      (0,-1.81,1.90651) -- (0,-1.8,1.8975) -- (0,-1.79,1.88851) -- (0,-1.78,1.87954) -- (0,-1.77,1.87057) -- (0,-1.76,1.86163) -- (0,-1.75,1.8527) -- (0,-1.74,1.84378) -- (0,-1.73,1.83489) --
      (0,-1.72,1.82601) -- (0,-1.71,1.81714) -- (0,-1.7,1.8083) -- (0,-1.69,1.79947) -- (0,-1.68,1.79065) -- (0,-1.67,1.78186) -- (0,-1.66,1.77309) -- (0,-1.65,1.76433) -- (0,-1.64,1.7556) --
      (0,-1.63,1.74688) -- (0,-1.62,1.73818) -- (0,-1.61,1.72951) -- (0,-1.6,1.72085) -- (0,-1.59,1.71222) -- (0,-1.58,1.7036) -- (0,-1.57,1.69501) -- (0,-1.56,1.68645) -- (0,-1.55,1.6779) --
      (0,-1.54,1.66938) -- (0,-1.53,1.66088) -- (0,-1.52,1.65241) -- (0,-1.51,1.64396) -- (0,-1.5,1.63553) -- (0,-1.49,1.62713) -- (0,-1.48,1.61876) -- (0,-1.47,1.61042) -- (0,-1.46,1.6021) --
      (0,-1.45,1.59381) -- (0,-1.44,1.58554) -- (0,-1.43,1.57731) -- (0,-1.42,1.56911) -- (0,-1.41,1.56093) -- (0,-1.4,1.55279) -- (0,-1.39,1.54467) -- (0,-1.38,1.53659) -- (0,-1.37,1.52854) --
      (0,-1.36,1.52053) -- (0,-1.35,1.51254) -- (0,-1.34,1.50459) -- (0,-1.33,1.49668) -- (0,-1.32,1.4888) -- (0,-1.31,1.48096) -- (0,-1.3,1.47315) -- (0,-1.29,1.46538) -- (0,-1.28,1.45765) --
      (0,-1.27,1.44996) -- (0,-1.26,1.44231) -- (0,-1.25,1.4347) -- (0,-1.24,1.42713) -- (0,-1.23,1.4196) -- (0,-1.22,1.41211) -- (0,-1.21,1.40467) -- (0,-1.2,1.39727) -- (0,-1.19,1.38992) --
      (0,-1.18,1.38261) -- (0,-1.17,1.37535) -- (0,-1.16,1.36814) -- (0,-1.15,1.36098) -- (0,-1.14,1.35386) -- (0,-1.13,1.3468) -- (0,-1.12,1.33978) -- (0,-1.11,1.33282) -- (0,-1.1,1.32591) --
      (0,-1.09,1.31905) -- (0,-1.08,1.31225) -- (0,-1.07,1.30551) -- (0,-1.06,1.29882) -- (0,-1.05,1.29219) -- (0,-1.04,1.28561) -- (0,-1.03,1.2791) -- (0,-1.02,1.27264) -- (0,-1.01,1.26625) -- (0,-1,1.25992)
      -- (0,-.99,1.25365) -- (0,-.98,1.24745) -- (0,-.97,1.24131) -- (0,-.96,1.23524) -- (0,-.95,1.22923) -- (0,-.94,1.22329) -- (0,-.93,1.21742) -- (0,-.92,1.21162) -- (0,-.91,1.20589) -- (0,-.9,1.20023) --
      (0,-.89,1.19464) -- (0,-.88,1.18913) -- (0,-.87,1.18369) -- (0,-.86,1.17833) -- (0,-.85,1.17304) -- (0,-.84,1.16783) -- (0,-.83,1.16269) -- (0,-.82,1.15763) -- (0,-.81,1.15266) -- (0,-.8,1.14776) --
      (0,-.79,1.14294) -- (0,-.78,1.1382) -- (0,-.77,1.13355) -- (0,-.76,1.12898) -- (0,-.75,1.12449) -- (0,-.74,1.12008) -- (0,-.73,1.11576) -- (0,-.72,1.11152) -- (0,-.71,1.10736) -- (0,-.7,1.1033) --
      (0,-.69,1.09931) -- (0,-.68,1.09542) -- (0,-.67,1.09161) -- (0,-.66,1.08788) -- (0,-.65,1.08425) -- (0,-.64,1.08069) -- (0,-.63,1.07723) -- (0,-.62,1.07385) -- (0,-.61,1.07056) -- (0,-.6,1.06736) --
      (0,-.59,1.06424) -- (0,-.58,1.06121) -- (0,-.57,1.05827) -- (0,-.56,1.05541) -- (0,-.55,1.05264) -- (0,-.54,1.04995) -- (0,-.53,1.04735) -- (0,-.52,1.04483) -- (0,-.51,1.04239) -- (0,-.5,1.04004) --
      (0,-.49,1.03777) -- (0,-.48,1.03558) -- (0,-.47,1.03347) -- (0,-.46,1.03145) -- (0,-.45,1.0295) -- (0,-.44,1.02762) -- (0,-.43,1.02583) -- (0,-.42,1.02411) -- (0,-.41,1.02247) -- (0,-.4,1.02089) --
      (0,-.39,1.01939) -- (0,-.38,1.01797) -- (0,-.37,1.01661) -- (0,-.36,1.01532) -- (0,-.35,1.01409) -- (0,-.34,1.01293) -- (0,-.33,1.01184) -- (0,-.32,1.01081) -- (0,-.31,1.00983) -- (0,-.3,1.00892) --
      (0,-.29,1.00806) -- (0,-.28,1.00726) -- (0,-.27,1.00652) -- (0,-.26,1.00582) -- (0,-.25,1.00518) -- (0,-.24,1.00459) -- (0,-.23,1.00404) -- (0,-.22,1.00354) -- (0,-.21,1.00308) -- (0,-.2,1.00266) --
      (0,-.19,1.00228) -- (0,-.18,1.00194) -- (0,-.17,1.00163) -- (0,-.16,1.00136) -- (0,-.15,1.00112) -- (0,-.14,1.00091) -- (0,-.13,1.00073) -- (0,-.12,1.00058) -- (0,-.11,1.00044) -- (0,-.1,1.00033) --
      (0,-.09,1.00024) -- (0,-.08,1.00017) -- (0,-.07,1.00011) -- (0,-.06,1.00007) -- (0,-.05,1.00004) -- (0,-.04,1.00002) -- (0,-.03,1.00001) -- (0,-.02,1) -- (0,-.01,1) -- (0,1.64105e-15,1) -- (0,.01,1) --
      (0,.02,.999997) -- (0,.03,.999991) -- (0,.04,.999979) -- (0,.05,.999958) -- (0,.06,.999928) -- (0,.07,.999886) -- (0,.08,.999829) -- (0,.09,.999757) -- (0,.1,.999667) -- (0,.11,.999556) --
      (0,.12,.999424) -- (0,.13,.999267) -- (0,.14,.999084) -- (0,.15,.998874) -- (0,.16,.998633) -- (0,.17,.99836) -- (0,.18,.998052) -- (0,.19,.997708) -- (0,.2,.997326) -- (0,.21,.996903) --
      (0,.22,.996438) -- (0,.23,.995928) -- (0,.24,.995371) -- (0,.25,.994764) -- (0,.26,.994107) -- (0,.27,.993395) -- (0,.28,.992628) -- (0,.29,.991803) -- (0,.3,.990918) -- (0,.31,.989969) --
      (0,.32,.988956) -- (0,.33,.987875) -- (0,.34,.986723) -- (0,.35,.985499) -- (0,.36,.9842) -- (0,.37,.982822) -- (0,.38,.981364) -- (0,.39,.979823) -- (0,.4,.978195) -- (0,.41,.976477) -- (0,.42,.974668)
      -- (0,.43,.972763) -- (0,.44,.970759) -- (0,.45,.968653) -- (0,.46,.966441) -- (0,.47,.96412) -- (0,.48,.961687) -- (0,.49,.959137) -- (0,.5,.956466) -- (0,.51,.95367) -- (0,.52,.950744) --
      (0,.53,.947685) -- (0,.54,.944487) -- (0,.55,.941146) -- (0,.56,.937655) -- (0,.57,.93401) -- (0,.58,.930205) -- (0,.59,.926232) -- (0,.6,.922087) -- (0,.61,.917762) -- (0,.62,.913249) --
      (0,.63,.908541) -- (0,.64,.90363) -- (0,.65,.898506) -- (0,.66,.89316) -- (0,.67,.887581) -- (0,.68,.88176) -- (0,.69,.875683) -- (0,.7,.869338) -- (0,.71,.86271) -- (0,.72,.855786) -- (0,.73,.848548)
      -- (0,.74,.840978) -- (0,.75,.833055) -- (0,.76,.824759) -- (0,.77,.816064) -- (0,.78,.806944) -- (0,.79,.797367) -- (0,.8,.787299) -- (0,.81,.776703) -- (0,.82,.765532) -- (0,.83,.753737) --
      (0,.84,.741259) -- (0,.85,.728029) -- (0,.86,.713967) -- (0,.87,.698976) -- (0,.88,.68294) -- (0,.89,.665716) -- (0,.9,.647127) -- (0,.91,.626947) -- (0,.92,.604879) -- (0,.93,.580526) --
      (0,.94,.553331) -- (0,.95,.522475) -- (0,.96,.486666) -- (0,.97,.443659) -- (0,.98,.388877) -- (0,.99,.309688) -- (0,1,-.0000188157) -- (0,1.01,-.311759) -- (0,1.02,-.394097) -- (0,1.03,-.452622) --
      (0,1.04,-.499819) -- (0,1.05,-.540184) -- (0,1.06,-.575913) -- (0,1.07,-.608259) -- (0,1.08,-.638015) -- (0,1.09,-.665715) -- (0,1.1,-.69174) -- (0,1.11,-.71637) -- (0,1.12,-.73982) -- (0,1.13,-.762256)
      -- (0,1.14,-.783812) -- (0,1.15,-.804596) -- (0,1.16,-.824696) -- (0,1.17,-.844188) -- (0,1.18,-.863133) -- (0,1.19,-.881584) -- (0,1.2,-.899588) -- (0,1.21,-.917185) -- (0,1.22,-.934408) --
      (0,1.23,-.951288) -- (0,1.24,-.967852) -- (0,1.25,-.984124) -- (0,1.26,-1.00013) -- (0,1.27,-1.01587) -- (0,1.28,-1.03139) -- (0,1.29,-1.04668) -- (0,1.3,-1.06177) -- (0,1.31,-1.07667) --
      (0,1.32,-1.09138) -- (0,1.33,-1.10593) -- (0,1.34,-1.12031) -- (0,1.35,-1.13454) -- (0,1.36,-1.14863) -- (0,1.37,-1.16258) -- (0,1.38,-1.17641) -- (0,1.39,-1.19011) -- (0,1.4,-1.20369) --
      (0,1.41,-1.21717) -- (0,1.42,-1.23053) -- (0,1.43,-1.2438) -- (0,1.44,-1.25697) -- (0,1.45,-1.27005) -- (0,1.46,-1.28304) -- (0,1.47,-1.29595) -- (0,1.48,-1.30878) -- (0,1.49,-1.32153) --
      (0,1.5,-1.3342) -- (0,1.51,-1.34681) -- (0,1.52,-1.35934) -- (0,1.53,-1.37181) -- (0,1.54,-1.38422) -- (0,1.55,-1.39657) -- (0,1.56,-1.40886) -- (0,1.57,-1.42109) -- (0,1.58,-1.43327) --
      (0,1.59,-1.4454) -- (0,1.6,-1.45747) -- (0,1.61,-1.4695) -- (0,1.62,-1.48148) -- (0,1.63,-1.49342) -- (0,1.64,-1.50531) -- (0,1.65,-1.51715) -- (0,1.66,-1.52896) -- (0,1.67,-1.54073) --
      (0,1.68,-1.55246) -- (0,1.69,-1.56415) -- (0,1.7,-1.57581) -- (0,1.71,-1.58743) -- (0,1.72,-1.59902) -- (0,1.73,-1.61057) -- (0,1.74,-1.62209) -- (0,1.75,-1.63358) -- (0,1.76,-1.64504) --
      (0,1.77,-1.65648) -- (0,1.78,-1.66788) -- (0,1.79,-1.67926) -- (0,1.8,-1.69061) -- (0,1.81,-1.70193) -- (0,1.82,-1.71323) -- (0,1.83,-1.7245) -- (0,1.84,-1.73575) -- (0,1.85,-1.74697) --
      (0,1.86,-1.75818) -- (0,1.87,-1.76936) -- (0,1.88,-1.78052) -- (0,1.89,-1.79166) -- (0,1.9,-1.80277) -- (0,1.91,-1.81387) -- (0,1.92,-1.82495) -- (0,1.93,-1.83601) -- (0,1.94,-1.84705) --
      (0,1.95,-1.85807) -- (0,1.96,-1.86908) -- (0,1.97,-1.88007) -- (0,1.98,-1.89104) -- (0,1.99,-1.90199) -- (0,2,-1.91293);

\end{scope}

\begin{scope}[shift={(4.2,-16.75)},scale=0.5,tdplot_main_coords,rotate=90]



  \filldraw[fill=black!20, draw=none, opacity=0.25]
    (-2.1,0,-2.1) -- (-2.1,0,2.1) -- (2.1,0,2.1) -- (2.1,0,-2.1) -- cycle;

    \filldraw[fill=black!20, draw=none, opacity=0.25]
    (-2.1,-2.1,0) -- (-2.1,2.1,0) -- (2.1,2.1,0) -- (2.1,-2.1,0) -- cycle;

  \filldraw[fill=black!20, draw=none, opacity=0.25]
    (0,-2.1,-2.1) -- (0,2.1,-2.1) -- (0,2.1,2.1) -- (0,-2.1,2.1) -- cycle;

\draw[] (-2.5,0,0) -- (2.5,0,0); 
  \draw[] (0,-2.5,0) -- (0,2.5,0); 
  \draw[] (0,0,-2.5) -- (0,0,2.5); 

 \filldraw[red] (0,1,0) circle (1pt);
 
 \draw[thick, red]  (-1.25992,1.5874,-2) -- (-1.20992,1.46391,-1.77121) -- (-1.15992,1.34542,-1.56058) -- (-1.10992,1.23192,-1.36734) -- (-1.05992,1.12343,-1.19075) -- (-1.00992,1.01994,-1.03006) --
      (-.959921,.921448,-.884518) -- (-.909921,.827956,-.753375) -- (-.859921,.739464,-.635881) -- (-.809921,.655972,-.531286) -- (-.759921,.57748,-.438839) -- (-.709921,.503988,-.357792) --
      (-.659921,.435496,-.287393) -- (-.609921,.372004,-.226893) -- (-.559921,.313512,-.175542) -- (-.509921,.260019,-.132589) -- (-.459921,.211527,-.0972859) -- (-.409921,.168035,-.0688812) --
      (-.359921,.129543,-.0466253) -- (-.309921,.0960511,-.0297682) -- (-.259921,.067559,-.01756) -- (-.209921,.0440668,-.00925056) -- (-.159921,.0255747,-.00408994) -- (-.109921,.0120826,-.00132814) --
      (-.059921,.00359053,-.000215148) -- (-.00992105,.0000984272,-9.76501e-7) -- (.040079,.00160632,.0000643797) -- (.090079,.00811422,.00073092) -- (.140079,.0196221,.00274864) -- (.190079,.03613,.00686755)
      -- (.240079,.0576379,.0138376) -- (.290079,.0841458,.0244089) -- (.340079,.115654,.0393314) -- (.390079,.152162,.059355) -- (.440079,.193669,.0852299) -- (.490079,.240177,.117706) --
      (.540079,.291685,.157533) -- (.590079,.348193,.205461) -- (.640079,.409701,.262241) -- (.690079,.476209,.328622) -- (.740079,.547717,.405354) -- (.790079,.624225,.493187) -- (.840079,.705733,.592871) --
      (.890079,.792241,.705157) -- (.940079,.883748,.830793) -- (.990079,.980256,.970531) -- (1.04008,1.08176,1.12512) -- (1.09008,1.18827,1.29531) -- (1.14008,1.29978,1.48185) -- (1.19008,1.41629,1.68549) --
      (1.24008,1.5378,1.90699);

\end{scope}

\begin{scope}[shift={(7.2,-16.75)},scale=0.5,tdplot_main_coords,rotate=90]



  \filldraw[fill=black!20, draw=none, opacity=0.25]
    (-2.1,0,-2.1) -- (-2.1,0,2.1) -- (2.1,0,2.1) -- (2.1,0,-2.1) -- cycle;

    \filldraw[fill=black!20, draw=none, opacity=0.25]
    (-2.1,-2.1,0) -- (-2.1,2.1,0) -- (2.1,2.1,0) -- (2.1,-2.1,0) -- cycle;

  \filldraw[fill=black!20, draw=none, opacity=0.25]
    (0,-2.1,-2.1) -- (0,2.1,-2.1) -- (0,2.1,2.1) -- (0,-2.1,2.1) -- cycle;

\draw[] (-2.5,0,0) -- (2.5,0,0); 
  \draw[] (0,-2.5,0) -- (0,2.5,0); 
  \draw[] (0,0,-2.5) -- (0,0,2.5); 

 \draw[thick, red]  (-1,1,-2) -- (-1,1,2);

\draw[thick, red]  (0,-1.73205,-2) -- (0,-1.68205,-1.95686) -- (0,-1.63205,-1.91405) -- (0,-1.58205,-1.8716) -- (0,-1.53205,-1.82953) -- (0,-1.48205,-1.78787) -- (0,-1.43205,-1.74665) -- (0,-1.38205,-1.70589) --
      (0,-1.33205,-1.66564) -- (0,-1.28205,-1.62593) -- (0,-1.23205,-1.5868) -- (0,-1.18205,-1.5483) -- (0,-1.13205,-1.51048) -- (0,-1.08205,-1.47338) -- (0,-1.03205,-1.43706) -- (0,-.982051,-1.40158) --
      (0,-.932051,-1.36701) -- (0,-.882051,-1.33342) -- (0,-.832051,-1.30089) -- (0,-.782051,-1.26949) -- (0,-.732051,-1.23931) -- (0,-.682051,-1.21045) -- (0,-.632051,-1.183) -- (0,-.582051,-1.15706) --
      (0,-.532051,-1.13273) -- (0,-.482051,-1.11012) -- (0,-.432051,-1.08934) -- (0,-.382051,-1.0705) -- (0,-.332051,-1.05369) -- (0,-.282051,-1.03902) -- (0,-.232051,-1.02657) -- (0,-.182051,-1.01644) --
      (0,-.132051,-1.00868) -- (0,-.0820508,-1.00336) -- (0,-.0320508,-1.00051) -- (0,.0179492,-1.00016) -- (0,.0679492,-1.00231) -- (0,.117949,-1.00693) -- (0,.167949,-1.01401) -- (0,.217949,-1.02348) --
      (0,.267949,-1.03528) -- (0,.317949,-1.04933) -- (0,.367949,-1.06555) -- (0,.417949,-1.08383) -- (0,.467949,-1.10407) -- (0,.517949,-1.12618) -- (0,.567949,-1.15003) -- (0,.617949,-1.17553) --
      (0,.667949,-1.20256) -- (0,.717949,-1.23104) -- (0,.767949,-1.26085) -- (0,.817949,-1.29191) -- (0,.867949,-1.32414) -- (0,.917949,-1.35744) -- (0,.967949,-1.39173) -- (0,1.01795,-1.42696) --
      (0,1.06795,-1.46305) -- (0,1.11795,-1.49994) -- (0,1.16795,-1.53756) -- (0,1.21795,-1.57588) -- (0,1.26795,-1.61484) -- (0,1.31795,-1.65439) -- (0,1.36795,-1.69449) -- (0,1.41795,-1.7351) --
      (0,1.46795,-1.7762) -- (0,1.51795,-1.81774) -- (0,1.56795,-1.85969) -- (0,1.61795,-1.90204) -- (0,1.66795,-1.94475) -- (0,1.71795,-1.9878);
 \draw[thick, red]  (0,-1.73205,2) -- (0,-1.68205,1.95686) -- (0,-1.63205,1.91405) -- (0,-1.58205,1.8716) -- (0,-1.53205,1.82953) -- (0,-1.48205,1.78787) -- (0,-1.43205,1.74665) -- (0,-1.38205,1.70589) --
      (0,-1.33205,1.66564) -- (0,-1.28205,1.62593) -- (0,-1.23205,1.5868) -- (0,-1.18205,1.5483) -- (0,-1.13205,1.51048) -- (0,-1.08205,1.47338) -- (0,-1.03205,1.43706) -- (0,-.982051,1.40158) --
      (0,-.932051,1.36701) -- (0,-.882051,1.33342) -- (0,-.832051,1.30089) -- (0,-.782051,1.26949) -- (0,-.732051,1.23931) -- (0,-.682051,1.21045) -- (0,-.632051,1.183) -- (0,-.582051,1.15706) --
      (0,-.532051,1.13273) -- (0,-.482051,1.11012) -- (0,-.432051,1.08934) -- (0,-.382051,1.0705) -- (0,-.332051,1.05369) -- (0,-.282051,1.03902) -- (0,-.232051,1.02657) -- (0,-.182051,1.01644) --
      (0,-.132051,1.00868) -- (0,-.0820508,1.00336) -- (0,-.0320508,1.00051) -- (0,.0179492,1.00016) -- (0,.0679492,1.00231) -- (0,.117949,1.00693) -- (0,.167949,1.01401) -- (0,.217949,1.02348) --
      (0,.267949,1.03528) -- (0,.317949,1.04933) -- (0,.367949,1.06555) -- (0,.417949,1.08383) -- (0,.467949,1.10407) -- (0,.517949,1.12618) -- (0,.567949,1.15003) -- (0,.617949,1.17553) --
      (0,.667949,1.20256) -- (0,.717949,1.23104) -- (0,.767949,1.26085) -- (0,.817949,1.29191) -- (0,.867949,1.32414) -- (0,.917949,1.35744) -- (0,.967949,1.39173) -- (0,1.01795,1.42696) --
      (0,1.06795,1.46305) -- (0,1.11795,1.49994) -- (0,1.16795,1.53756) -- (0,1.21795,1.57588) -- (0,1.26795,1.61484) -- (0,1.31795,1.65439) -- (0,1.36795,1.69449) -- (0,1.41795,1.7351) -- (0,1.46795,1.7762)
      -- (0,1.51795,1.81774) -- (0,1.56795,1.85969) -- (0,1.61795,1.90204) -- (0,1.66795,1.94475) -- (0,1.71795,1.9878);
\end{scope}

\begin{scope}[xscale=0.6,yscale=-1]
    
        \node at (0.1,17.9) [rotate=90] {$\mathcal{C}_1$};
        \node at (5.1,17.9) [rotate=90] {$\mathcal{C}_2$};
        \node at (10.1,17.9) [rotate=90] {$\mathcal{C}_3$};

\begin{scope}[shift={(0,1)}]
\node (0) at (-4,0) [rotate=90] {$R$};
\node (1) at (-4,5) [rotate=90] {$R(-2)^2 \oplus R(-3)^2$};
\draw [->] (1) --node[rotate=90,above]{\tiny $\partial_0=\left[f_{x_3^2}\ f_{x_3x_2} \ f_{x_2^3} \ f_{x_3x_1^2}\right]$} (0);
\node (2) at (-4,9) [rotate=90] {$R(-3) \oplus R(-4)^3$};
\node (3) at (-4,12) [rotate=90] {$R(-5)$};
\node (4) at (-4,13.5) [rotate=90] {$0$};

\draw [->] (2) --node[rotate=90,above]{\tiny $\partial_1$} (1);
\draw [->] (3) --node[rotate=90,above]{\tiny $\partial_2$} (2);
\draw [->] (4) -- (3);
\end{scope}

        \node at (-1,2.5) [rotate=90] {$\partial_0$ }; 
        \node at (-1,9) [rotate=90] {$\partial_1$};         
        \node at (-1,14) [rotate=90] {$\partial_2$};  

\draw (-1.5,12.75) -- (14.5,12.75);
\draw (-1.5,15.25) -- (14.5,15.25);
\draw (-1.5,5.25) -- (14.5,5.25);
\draw (-1.5,0) -- (14.5,0);
\draw (-0.5,18.25) -- (14.5,18.25);

\draw (-1.5,0) -- (-1.5,15.25);
\draw (-0.5,0) -- (-0.5,18.25);
\draw (4.5,0) -- (4.5,18.25);
\draw (9.5,0) -- (9.5,18.25);
\draw (14.5,0) -- (14.5,18.25);

        \node at (0.5,2.5) [rotate=90] {\parbox{4cm}{\tiny $\uwave{x_3^2} - x_3x_0$}};
        \node at (1.5,2.5) [rotate=90] {\parbox{4cm}{\tiny $\uwave{x_3x_2} + x_3x_0$}};
        \node at (2.5,2.5) [rotate=90] {\parbox{4cm}{\tiny $\uwave{x_2^3} - x_3x_1x_0+2x_3x_0^2 + x_1^3 - x_0^3$}};
        \node at (3.5,2.5) [rotate=90] {\parbox{4cm}{\tiny $\uwave{x_3x_1^2} - x_3x_0^2$}};
        
        \node at (5.5,2.5) [rotate=90] {\parbox{4cm}{\tiny $\uwave{x_3^2} - x_3x_1 + x_2^2 - x_2x_0$}};
        \node at (6.5,2.5) [rotate=90] {\parbox{4cm}{\tiny $\uwave{x_3x_2} - x_1x_0$}};
        \node at (7.5,2.5) [rotate=90] {\parbox{4cm}{\tiny $\uwave{x_2^3} + x_3x_1x_0 - x_2^2 x_0 -x_1^2x_0$}};
        \node at (8.5,2.5) [rotate=90] {\parbox{4cm}{\tiny $\uwave{x_3x_1^2} - x_2^2x_1$}};
        
        \node at (10.5,2.5) [rotate=90] {\parbox{4cm}{\tiny $\uwave{x_3^2}+x_{3}x_{0}$}};
        \node at (11.5,2.5) [rotate=90] {\parbox{4cm}{\tiny $\uwave{x_3x_2} -x_{3}x_{0} $}};
        \node at (12.5,2.5) [rotate=90] {\parbox{4cm}{\tiny $\uwave{x_2^3} -x_{2}x_{1}^{2}-x_{2}^{2}x_{0}+x_{1}^{2}x_{0}+x_{2}x_{0}^{2}-x_{0}^{3}$}};
        \node at (13.5,2.5) [rotate=90] {\parbox{4cm}{\tiny $\uwave{x_3x_1^2} -x_{2}^{2}x_{0}-2\,x_{3}x_{0}^{2}+x_{1}^{2}x_{0}-x_{0}^{3}$}};

\node at (2,9) [rotate=90] {\tiny $\left[\begin{array}{cccc}
      -x_{2}-x_{0}&x_{1}x_{0}-2\,x_{0}^{2}&-x_{1}^{2}+x_{0}^{2}&0\\
      \uwave{x_{3}}-x_{0}&-x_{2}^{2}+x_{2}x_{0}-x_{0}^{2}&0 &-x_{1}^{2}+x_{0}^{2}\\
      0&\uwave{x_{3}}&0&0\\
      0&-x_{1}&\uwave{x_{3}}-x_{0}&\uwave{x_{2}}+x_{0}
      \end{array}\right]$};

      \node at (2,14) [rotate=90] {\tiny $\left[\begin{array}{c}
      x_1^2 - x_0^2 \\ 0 \\ -x_2 - x_0 \\ \uwave{x_3} - x_0
      \end{array}\right]$};

\node at (7,9) [rotate=90] {\tiny $\left[\begin{array}{cccc}
      -x_{2}&-x_{1}x_{0}&-x_{1}^{2}&0\\
      \uwave{x_{3}}-x_{1}&-x_{2}^{2}+x_{2}x_{0}&x_{2}x_{1}&-x_{1}^{2}\\
      1&\uwave{x_{3}}&0&x_{1}\\
      0&0&\uwave{x_{3}}-x_{1}&\uwave{x_{2}}-x_{0}
      \end{array}\right]$};

\node at (7,14) [rotate=90] {\tiny $\left[\begin{array}{c}
      x_{1}^{2}\\
      -x_{1}\\
      -x_{2}+x_{0}\\
      \uwave{x_{3}}-x_{1}
      \end{array}\right]$};

\node at (12,9) [rotate=90] {\tiny $\left[\begin{array}{cccc}
      -x_{2}+x_{0}&0&-x_{1}^{2}+2\,x_{0}^{2}&0\\
      \uwave{x_{3}}+x_{0}&-x_{2}^{2}+x_{1}^{2}-x_{0}^{2}&x_{2}x_{0}+x_{0}^{2}&-x_{1}^{2}+2\,x_{0}^{2}\\
      0&\uwave{x_{3}}&0&x_{0}\\
      0&0&\uwave{x_{3}}&\uwave{x_{2}}-x_{0}
      \end{array}\right]$};

      \node at (12,14) [rotate=90] {\tiny $\left[\begin{array}{c}
      x_{1}^{2}-2x_0^2\\
      -x_{0}\\
      -x_{2}+x_{0}\\
      \uwave{x_{3}}
      \end{array}\right]$};

\end{scope}

        \end{tikzpicture}
    \end{center}
    \caption{Examples of $J$-marked bases defining points in the irreducible components of $\mathbf{Hilb}^{3t+2}(\mathbb{P}^3)$ with their $J$-resolutions. The computation of $J$-resolutions is available in the \textit{M2} ancillary file \href{http://www.paololella.it/software/examples-curves-each-component.m2}{\tt examples- curves-each-component.m2}}
    \label{fig:example}
\end{figure}

Among the saturated (strongly) stable ideals in $\mathbb{K}[x_0,x_1,x_2,x_3]$ that define schemes param-etrized by $\mathbf{Hilb}^{3t+2}(\mathbb{P}^3)$, we consider
\[
J = (x_3^2,x_3 x_2, x_2^3, x_3 x_1^2) \subset \mathbb{K}[x_0,x_1,x_2,x_3].
\]

The marked scheme $\mathbf{Mf}_J$ is an open subscheme of $\mathbf{Hilb}^{3t+2}(\mathbb{P}^3)$  by \cite[Corollary 4.3]{LR2}. Hence, the same holds for $\MSyzScheme{J}$ and $\MResScheme{J}$ by Corollary \ref{cor:isomorfismo}. Moreover, these subschemes are isomorphic to a dense open subset of $\mathbf{Hilb}^{3t+2}(\mathbb{P}^3)$, because there exist $J$-marked bases defining smooth points on each irreducible component. For example:
\begin{enumerate}
    \item the ideal
    \[
    I_1 = (x_3,x_2^3 + x_1^3 - x_0^3) \cap (x_3-x_0,x_2+x_0,x_1-x_0) \cap (x_3-x_0,x_2+x_0,x_1+x_0) 
    \]
    defines a smooth point in $\mathcal{C}_1 \cap \mathbf{Mf}_J$;
    \item the ideal
    \[
    I_2 = (x_3^2 - x_2x_0, x_3x_2 - x_1x_0,  x_3x_1 - x_2^2) \cap (x_3,x_2-x_0,x_1) 
    \]
    defines a smooth point in $\mathcal{C}_2 \cap \mathbf{Mf}_J$;
    \item the ideal
    \[
    I_3 = (x_3,x_2^2 - x_1^2 +x_0^2) \cap (x_3+x_0,x_2-x_0) 
    \]
    defines a smooth point in $\mathcal{C}_3 \cap \mathbf{Mf}_J$.
\end{enumerate}

The $J$-marked bases and the  $J$-resolutions of $I_1$, $I_2$ and $I_3$ are shown in Figure \ref{fig:example}. 
To verify the smoothness of the given points, one must compute the dimension of the tangent space to $\mathbf{Hilb}^{3t+2}(\mathbb{P}^3)$ at each point and check that it equals the dimension of the corresponding irreducible component.

\smallskip

As a consequence of density, the ideals defining $\mathbf{Mf}_J$, $\MSyzScheme{J}$, and $\MResScheme{J}$ admit a primary decomposition consisting of three components. In principle, the primary decomposition can be computated directly; however, due to the large number of parameters involved, such computations are highly resource-intensive (CPU time and RAM memory) and quickly become impractical.

Nonetheless, by combining theoretical insights with computational techniques, we can exploit the algebraic and geometric properties of generic elements of each component to predict the conditions that characterize them. These conditions can be used to parametrize an open subset of each component effectively.

\medskip

First, let us consider the universal family parametrized by $\MResScheme{J}$. To simplify the notation, we denote by \lq$C$\rq~ the parameters used to describe the generators, by \lq$B$\rq~ the parameters used to describe the first syzygies, and by \lq$A$\rq~ the parameters used to describe the second syzygies. Let $\mathcal{R}$ be the polynomial ring $\kk[C,B,A][x_0,x_1,x_2,x_3]$, and consider the complex
\[
0 \xrightarrow{} \mathcal R(-5) \xrightarrow{\partial_2}\mathcal R(-3)  \oplus \mathcal R(-4)^3   \xrightarrow{\partial_1} \mathcal R(-2)^2\oplus \mathcal R(-3)^2 \xrightarrow{\partial_0}\mathcal I
\xrightarrow{} 0,
\]
where
\begin{itemize}
    \item $\partial_0 = \left[\begin{array}{cccc}  f_{x_3^2} & f_{x_3 x_2} & f_{x_2^3} & f_{x_3 x_1^2}\end{array}\right]$ describes the $J$-marked set 
 \begin{equation}\label{eq:finExMSet}{\tiny
    \begin{split}
    f_{x_3^2} = {}&  \uwave{x_{3}^{2}}+C_{0}x_{3}x_{1}+C_{1}x_{2}^{2}+C_{2}x_{3}x_{0}+C_{3}x_{2}x_{1}+C_{4}x_{1}^{2}+C_{5}x_{2}x_{0}+C_{6}x_{1}x_{0}+C_{7}x_{0}^{2},\\
f_{x_3 x_2} = {} & \uwave{x_{3}x_{2}}+C_{8}x_{3}x_{1}+C_{9}x_{2}^{2}+C_{10}x_{3}x_{0}+C_{11}x_{2}x_{1}+C_{12}x_{1}^{2}+C_{13}x_{2}x_{0}+C_{14}x_{1}x_{0}+C_{15}x_{0}^{2},\\
f_{x_2^3} = {}&\uwave{x_{2}^{3}}+C_{16}x_{2}^{2}x_{1}+C_{17}x_{3}x_{1}x_{0}+C_{18}x_{2}x_{1}^{2}+C_{19}x_{2}^{2}x_{0}+C_{20}x_{3}x_{0}^{2}+C_{21}x_{1}^{3}+C_{22}x_{2}x_{1}x_{0}+C_{23}x_{1}^{2}x_{0}+C_{24}x_{2}x_{0}^{2}+C_{25}x_{1}x_{0}^{2}+C_{26}x_{0}^{3},\\
f_{x_3 x_1^2} = {}&\uwave{x_{3}x_{1}^{2}}+C_{27}x_{2}^{2}x_{1}+C_{28}x_{3}x_{1}x_{0}+C_{29}x_{2}x_{1}^{2}+C_{30}x_{2}^{2}x_{0}+C_{31}x_{3}x_{0}^{2}+C_{32}x_{1}^{3}+C_{33}x_{2}x_{1}x_{0}+C_{34}x_{1}^{2}x_{0}+C_{35}x_{2}x_{0}^{2}+C_{36}x_{1}x_{0}^{2}+C_{37}x_{0}^{3}\\
    \end{split}}
    \end{equation}
    \item $\partial_1$ describes the marked set of first syzygies
    \begin{equation}\label{eq:delta1}{\tiny
        \left[\begin{array}{cccc}
      B_{0}x_{2}+B_{1}x_{1}+B_{2}x_{0}&(\partial_1)_{1,2}&(\partial_1)_{1,3}&B_{42}x_{1}^{2}+B_{43}x_{1}x_{0}+B_{44}x_{0}^{2}\\
      \uwave{x_{3}}+B_{3}x_{2}+B_{4}x_{1}+B_{5}x_{0}&(\partial_1)_{2,2}& (\partial_1)_{2,3}&B_{45}x_{1}^{2}+B_{46}x_{1}x_{0}+B_{47}x_{0}^{2}\\
      B_{6}&\uwave{x_{3}}+B_{20}x_{2}+B_{21}x_{1}+B_{22}x_{0}&B_{37}x_{2}+B_{38}x_{1}+B_{39}x_{0}&B_{48}x_{1}+B_{49}x_{0}\\
      B_{7}&B_{23}x_{1}+B_{24}x_{0}&\uwave{x_{3}}+B_{40}x_{1}+B_{41}x_{0}&\uwave{x_{2}}+B_{50}x_{1}+B_{51}x_{0}
      \end{array}\right]}
    \end{equation}
    where
    \[\tiny
    \left[\begin{array}{cc}(\partial_1)_{1,2} & (\partial_1)_{1,3} \\ (\partial_1)_{2,2} & (\partial_1)_{2,3}\end{array}\right] =
       \left[\begin{array}{cc} B_{8}x_{2}^{2}+B_{9}x_{2}x_{1}+B_{10}x_{1}^{2}+B_{11}x_{2}x_{0}+B_{12}x_{1}x_{0}+B_{13}x_{0}^{2} & B_{14}x_{2}^{2}+B_{15}x_{2}x_{1}+B_{16}x_{1}^{2}+B_{17}x_{2}x_{0}+B_{18}x_{1}x_{0}+B_{19}x_{0}^{2}\\
          B_{25}x_{2}^{2}+B_{26}x_{2}x_{1}+B_{27}x_{1}^{2}+B_{28}x_{2}x_{0}+B_{29}x_{1}x_{0}+B_{30}x_{0}^{2}& B_{31}x_{2}^{2}+B_{32}x_{2}x_{1}+B_{33}x_{1}^{2}+B_{34}x_{2}x_{0}+B_{35}x_{1}x_{0}+B_{36}x_{0}^{2}\end{array}\right]
    \]
    \item $\partial_2$ describes the marked set of second syzygies
    \begin{equation}\label{eq:delta2}{\tiny
        \left[\begin{array}{c}
      \vphantom{\left\{3\right\}}A_{0}x_{2}^{2}+A_{1}x_{2}x_{1}+A_{2}x_{1}^{2}+A_{3}x_{2}x_{0}+A_{4}x_{1}x_{0}+A_{5}x_{0}^{2}\\
      \vphantom{\left\{4\right\}}A_{6}x_{2}+A_{7}x_{1}+A_{8}x_{0}\\
      \vphantom{\left\{4\right\}}A_{9}x_{2}+A_{10}x_{1}+A_{11}x_{0}\\
      \vphantom{\left\{4\right\}}\uwave{x_{3}}+A_{12}x_{2}+A_{13}x_{1}+A_{14}x_{0}
      \end{array}\right]}
    \end{equation}
\end{itemize}
The support of each polynomial appearing in \eqref{eq:delta1} and \eqref{eq:delta2}  is defined according to Corollary~\ref{cor:suppSyz}. 
We obtain the polynomials defining the scheme $\MResScheme{J}$ by imposing the exactness of the complex. Let $\mathscr{S}^{(1)}$ and $\mathscr{S}^{(2)}$ denote the ideals generated by the conditions ensuring that $\partial_0 \circ \partial_1 = 0$ and $\partial_1 \circ \partial_2 = 0$, respectively. Then we have:
\begin{itemize}
    \item $\MResScheme{J} = \Spec\big( \kk[C,B,A]/(\mathscr{S}^{(1)} + \mathscr{S}^{(2)})\big) $ as in Proposition \ref{prop:elimSyzn};
    \item $\MSyzScheme{J} = \Spec\big( \kk[C,B]/\mathscr{S}^{(1)} \big) $ as in Theorem \ref{thm:reprFuncSyz};
    \item $\MFScheme{J} = \Spec\big( \kk[C]/\mathscr{U} \big)$ where $\mathscr{U} = \mathscr{S}^{(1)} \cap \kk[C]$ as in Corollary \ref{cor:eliminazione}.
\end{itemize}
The \textit{M2} ancillary file \href{http://www.paololella.it/software/equations-marked-schemes.m2}{\tt equations-marked-schemes.m2} provides the description of all ideals. The\hfill file\hfill \href{http://www.paololella.it/software/projectionFirstFactor-curves-P3.m2}{\tt projectionFirstFactor-curves-P3.m2}\hfill contains\hfill the\hfill description\hfill of\\ isomorphisms $\MResScheme{J} \to \MFScheme{J}$ and $\MSyzScheme{J} \to \MFScheme{J}$. Meanwhile, the file \href{http://www.paololella.it/software/projectionSecondFactor-curves-P3.m2}{\tt projectionSecond Factor-curves-P3.m2} contains the equations of the scheme $\MFScheme{J}^{\pi_2}$ (Theorem \ref{th:representation second factor}) and the description of the isomorphisms $\MResScheme{J} \to \MFScheme{J}^{\pi_2}$ and $\MSyzScheme{J} \to \MFScheme{J}^{\pi_2}$ (Theorem \ref{th:isomorphism 2}).


\medskip

\paragraph{\em First irreducible component.} The generic 1-dimensional scheme $Z_1$ parametrized by the component $\mathcal{C}_1$ is minimally generated by its $J$-marked basis, and the $J$-resolution is minimal. By semicontinuity of Betti numbers, the resolution remains minimal in an open neighborhood of $Z_1 \in \MResScheme{J}$. Consequently, the parameters $B_6$ and $B_7$ vanish on an open subset of $\mathcal{C}_1$ (compare \eqref{eq:delta1} and $\partial_1$ in the $J$-resolution in the left column of Figure \ref{fig:example}). 

The scheme $Z_1$ is not a locally Cohen-Macaulay (lCM for short) curve, as $Z_1$ has two isolated points. This geometric property is captured by the Hartshorne-Rao module
\[
M_{Z_1} = \mathsf{H}^1_{\bullet}(\mathcal{I}_{Z_1})=\bigoplus_{t \in \mathbb{Z}} \mathsf{H}^1 \big(\mathbb{P}^3,\mathcal{I}_{Z_1}(t)\big)
\]
(see \cite{Rao,MDP} for an introduction to this invariant). Indeed, a 1-dimensional scheme is a lCM curve if and only if its Hartshorne-Rao module is artinian. The Hartshorne-Rao module can be read off from the resolution as $(\textnormal{coker}\, \partial_2^T)^\ast$. In this case, up to a shift, $\textnormal{coker}\, \partial_2^T$ is a quotient ring with Hilbert function $(1,2,2,\ldots)$. The ideal generated by the entries of $\partial_2$ is
\[
(x_3 + A_{13}x_1 + A_{14}x_0,\ -x_2 + A_{10}x_1 + A_{11}x_0,\ A_7 x_1 + A_8 x_0,\ x_1^2 + A_4 x_1x_0 + A_5 x_0^2) 
\]
and the quotient has Hilbert function $(1,2,2,\ldots)$ if and only if $A_7 = A_8 = 0$.

Hence, the component $\mathcal{U}_1 \subseteq \MResScheme{J}$ birational to the irreducible component $\mathcal{C}_1 $ lies in the hyperplane defined by equations
\[
B_6 = B_7 = A_7 = A_8 = 0.
\]
The projection onto the second factor, $\pi_2(\mathcal{U}_1) \subset \MFScheme{J}^{\pi_2}$, is contained in the hyperplane defined by
\[
B_6 = B_7 = B_{48} = B_{49} = 0.
\]
Adding these four equations to those defining $\MFScheme{J}^{\pi_2}$, we obtain an open subset of $\mathcal{C}_1$ isomorphic to $\mathbb{A}^{18}$. Specifically, it is defined by the polynomials
\[
\begin{split}
 &B_{45}+1,\:B_{44},\:B_{43},\:B_{42},\:B_{37},\:B_{31},\:B_{28},\:B_{27}+1,\:B_{26},\:B_{25},\:B_{14}+1,\:B_{11},\:B_{10},\:B_{9},\:B_{8},\:B_{0}+1,\\
 &B_{48},\:B_{32},\:B_{29}-B_{46
      },\:B_{1}+B_{50},\:B_{49},\:B_{34},\:B_{30}-B_{47},\:B_{20}+B_{33},\:B_{3}-B_{33},\:B_{2}+B_{51},\:B_{38},\\
      &B_{35}+B_{33}B_{46},\:B_{4}-B_{40}-B_{33}B_{50},\:B_{39},\:B_{36}+B_{33}B_{47},\:B_{7},\:B_{6},\:B_{5}-B_{41}-B_{33}B_{51}.
\end{split}
\]

Similarly, for the projection onto the first factor, $\pi_1(\mathcal{U}_1) \subset \MFScheme{J}$, the scheme $\pi_1(\mathcal{U}_1)$ lies in the affine variety defined by
\[
C_1 + C_9^2 = C_{12} - C_8C_{11} - C_8^2C_9 = C_{27} = C_{30} = 0.
\]
The equations in this case are more complex and will not be included here. Readers interested in the full list of polynomials can consult the \textit{M2} ancillary file \href{http://www.paololella.it/software/rationalParametrization-planeCubic+twoPoints.m2}{\tt rationalParametrization- planeCubic+twoPoints.m2}.

\medskip

\paragraph{\em Second irreducible component.} For the generic 1-dimensional scheme $Z_2$ parametrized by the component $\mathcal{C}_2$, the corresponding $J$-resolution is not minimal, hence either $B_6 \neq 0$ or $B_7 \neq 0$. Moreover, the scheme $Z_2$ is not a locally Cohen-Macaualay curve, since it has one isolated point. The Hartshorne-Rao module $M_{Z_2}$ is therefore not artinian, and, up to a shift, $\textnormal{coker}\, \partial_2^T$ is a quotient ring with Hilbert function $(1,1,1,\ldots)$. In order to obtain such a Hilbert function, the following conditions must be satisfied:
\begin{itemize}
    \item $A_7 \neq 0$ or $A_8 \neq 0$ (otherwise the Hilbert function of the Hartshorne-Rao module would be $(1,2,2,\ldots)$);
    \item the polynomial $x_1^2 + A_4 x_1x_0 + A_5x_0^2$ must be a linear combination of $x_1(A_7x_1+A_8x_0)$ and $x_0(A_7x_1+A_8x_0)$, that is,
    \[
    \det\left[\begin{array}{ccc} 1 & A_4 & A_5 \\ A_7 & A_8 & 0 \\ 0 & A_7 & A_8\end{array}\right] = 0.
    \]
\end{itemize}

Hence, the component $\mathcal{U}_2 \subset \MResScheme{J}$ birational to the second irreducible component $\mathcal{C}_2$ is contained in the hypersurface defined by 
\[
    A_{5}A_{7}^{2}-A_{4}A_{7}A_{8}+A_{8}^{2} = 0
\]
and an open subset of $\mathcal{U}_2$ is contained in the open subset $\MResScheme{J}\setminus V(B_6,B_7,A_7,A_8)$.

Considering the projections onto the first and second factors, we obtain analogous inclusion relations expressed in terms of the parameters \lq$C$\rq and \lq$B$\rq.
In both cases, it follows that $\mathcal{U}_2$ is rational of dimension 15. The complete list of equations can be found in the \textit{M2} ancillary file \href{http://www.paololella.it/software/rationalParametrization-twistedCubic+Point.m2}{\tt rationalParametrization-twistedCubic+Point.m2}.

\medskip

\paragraph{\em Third irreducible component.} For the generic 1-dimensional scheme $Z_3$ parametrized by the component $\mathcal{C}_3$, the corresponding $J$-resolution is  minimal, and the Hartshorne-Rao module $M_{Z_3}$ is artinian. Hence, the component $\mathcal{U}_3 \subset \MResScheme{J}$ birational to the irreducible component $\mathcal{C}_3$ is contained in the hyperplane $V(B_6,B_7)$ and it has an open subset contained in the complement of
\[
V(A_7) \cup V(A_8) \cup V(A_{5}A_{7}^{2}-A_{4}A_{7}A_{8}+A_{8}^{2}).
\]
Again, considering the projections onto the first and second factors, we obtain analogous inclusion relations expressed in terms of the parameters \lq$C$\rq and \lq$B$\rq.
In both cases, one verifies that $\mathcal{U}_3$ is rational of dimension 12. The complete list of defining equations can be found in the \textit{M2} ancillary file \href{http://www.paololella.it/software/rationalParametrization-twistedCubic+Point.m2}{\tt rationalParametrization-planeConic+Line.m2}.



\providecommand{\bysame}{\leavevmode\hbox to3em{\hrulefill}\thinspace}
\providecommand{\MR}{\relax\ifhmode\unskip\space\fi MR }
\providecommand{\MRhref}[2]{%
  \href{http://www.ams.org/mathscinet-getitem?mr=#1}{#2}
}
\providecommand{\href}[2]{#2}

\end{document}